\documentclass[10pt]{amsart}

\usepackage{amsmath,xspace,amssymb,mathrsfs}
\usepackage{color}

\input xy
\xyoption{all}
\xyoption{2cell}
\UseAllTwocells
\CompileMatrices

\newcommand{\Spec}{\operatorname{Spec}}
\renewcommand{\phi}{\varphi}

\newcommand{\Ker}{\operatorname{Ker}}

\newcommand{\Supp}{\operatorname{Supp}}

\newcommand{\pdim}{\operatorname{purd}}
\newcommand{\ev}{\operatorname{ev}}

\newtheorem{proposition}{Proposition}[section]
\newtheorem{lemma}[proposition]{Lemma}

\newtheorem{corollary}[proposition]{Corollary}
\newtheorem{theorem}[proposition]{Theorem}

\theoremstyle{definition}

\newtheorem{example}[proposition]{Example}

\newtheorem{remark}[proposition]{Remark}

\usepackage{etoolbox}
\makeatletter
\patchcmd{\@settitle}{\uppercasenonmath\@title}{}{}{}
\patchcmd{\@setauthors}{\MakeUppercase}{}{}{}
\makeatother

\begin{document}

\title[Counterexamples to the maximality of purely-primes]{Pure dimension and counterexamples to the maximality of purely-primes}

\author[A. Tarizadeh]{Abolfazl Tarizadeh}
\address{Department of Mathematics, Faculty of Basic Sciences, University of Maragheh, Maragheh, East Azerbaijan Province, Iran.}
\email{ebulfez1978@gmail.com}

\date{}

\begin{abstract} In analogy with the notion of Krull dimension, for every commutative ring $R$, we introduce the notion of pure dimension $\pdim(R)$ by considering chains of purely-prime ideals. Then we see that the inequality $\pdim(R)\leqslant\dim(R)$ holds for a very broad class of rings $R$, including pseudo-Noetherian rings (e.g., Noetherian rings), Gelfand rings (e.g., local rings), mp-rings (e.g., integral domains), and rings with finitely many maximal ideals or finitely many minimal primes. In fact, for all of these rings $R$, we always have $\pdim(R)=0$, i.e., every purely-prime ideal of $R$ is purely-maximal.  \\
However, we prove a classification theorem which states for each ordinal number $\Omega$ the proper pure ideals of the commutative ring $$R_{\Omega}=k[X_{\alpha} : \alpha <\Omega]/\big(X_{\alpha}(1-X_{\beta}) : \alpha<\beta<\Omega\big)$$ 
are exactly of the form $I_{\lambda}=(x_{\alpha}: \alpha<\lambda)$ with $\lambda=0$ or $\lambda\leqslant\Omega$ a limit ordinal (we write $x_{\alpha}$ for the image of $X_{\alpha}$ in $R_{\Omega}$). In addition, all of them are purely-prime ideals and form the following unique strictly ascending chain: $$I_{0}=0\subset I_{\omega}\subset I_{\omega\cdot2}\subset I_{\omega\cdot 3}\subset\ldots\subset I_{\omega\cdot\alpha}$$
where $\alpha$ is a unique ordinal with $\Omega=\omega\cdot\alpha+n$ and $n\in\omega$.
To prove this theorem, we use the full strength of ordinal arithmetic as well as several new technical lemmas. As an application, the pure dimension of $R_{\Omega}$ is explicitly computed: \[\pdim(R_{\Omega})= 
\begin{cases}
    \alpha & \text{if } \alpha<\omega, \\
    \infty & \text{if } \alpha\geqslant\omega.
\end{cases}
\] 
In particular, if $\Omega\geqslant\omega\cdot2$ then we obtain that  $\pdim(R_{\Omega})\geqslant2>1=\dim(R_{\Omega})$ which strongly violates the inequality $\pdim(R)\leqslant\dim(R)$. If moreover, $\Omega\geqslant\omega^{2}$ we get that $\pdim(R_{\Omega})=\infty$. \\
The above classification theorem, in particular, shows that the commutative ring $$R_{\omega}=k[X_{0}, X_{1}, X_{2}, \ldots]/\big(X_{i}(1-X_{j}) : i<j\big)$$ has a purely-prime ideal that is not purely-maximal. This gives a counterexample to \cite[Conjecture 5.8]{Tarizadeh}. In fact, every ring $R_{\Omega}$ with $\Omega\geqslant\omega$ is a counterexample to this conjecture. \\ 
Next, we study some fundamental properties of pure dimension. In particular, for every commutative ring $R$, we prove that $$\pdim(R)=\pdim(R/N)=\pdim(R[x])=
\pdim(R[[x]])$$ with $N$ the nil-radical of $R$. \\ 
We also prove an important characterization result which states that a ring $R$ is pseudo-Noetherian (i.e., every pure ideal is finitely generated) if and only if every purely-prime ideal of $R$ is exactly of the form $R(1-e)$ with $e\in R$ a primitive idempotent, or equivalently, every purely-prime ideal of $R$ is purely-maximal and $R$ satisfies the ascending chain condition on pure ideals. 
\end{abstract}

\maketitle

\section{Introduction}

Pure ideals and purely-prime ideals provide a natural framework for studying certain aspects of the ideal structure of commutative rings. Recall that every purely-maximal ideal is purely-prime, but the converse need not hold in general. In the noncommutative setting, examples of this phenomenon are already known. In fact, in \cite[Chap. 7, Example 36]{Borceux}, a noncommutative ring is constructed that has a two-sided maximal ideal whose pure part is not a purely-maximal ideal. Consequently, there exist noncommutative rings having purely-prime ideals that are not purely-maximal.

However, in the commutative case, the construction of such an example is by no means straightforward. Indeed, the proposition ``every purely-prime ideal is purely-maximal'' is true for a very wide class of commutative rings. For instance, every purely-prime ideal of a Gelfand ring (e.g., local rings or zero-dimensional rings) is purely-maximal \cite[Chap. 8, Proposition 37]{Borceux}. The same conclusion holds for mp-rings (e.g., integral domains); see Corollary \ref{Coro 2 mp}. Moreover, every purely-prime ideal of a Noetherian ring and more generally a pseudo-Noetherian ring is purely-maximal (see Theorem \ref{Theorem 2}). The same conclusion holds for a ring with finitely many maximal ideals or finitely many minimal prime ideals (see Theorem \ref{Theorem 8 finite min-max}). Roughly speaking, the purely-prime ideals of any commutative ring we familiar with are purely-maximal. 

In view of the preceding results, constructing a commutative ring having a purely-prime ideal that is not purely-maximal appears to be a very difficult problem. This naturally led us to formulate the question as a conjecture in  \cite[Conjecture 5.8]{Tarizadeh}.  But the problem remained unresolved for several years. However, in Theorem \ref{Theorem 1}, we construct a commutative ring that provides a counterexample to this conjecture.  
The construction is quite natural, but establishing the precise structure of its pure ideals requires a more delicate argument. This counterexample not only settles the conjecture negatively but also motivates a systematic study of chains of purely-prime ideals. 

To this end, we introduce the pure dimension of a commutative ring $R$, denoted by $\pdim(R)$, as the supremum of the lengths of all finite strict chains of purely-prime ideals of $R$. This definition is analogous to the usual definition of Krull dimension, with purely-prime ideals replacing prime ideals. Although the two dimensions are defined in parallel, we will show that their behavior can differ substantially. The inequality $\pdim(R)\leqslant\dim(R)$ holds for many important classes of rings, but it does not hold in general.  

In fact, in Theorem \ref{Theorem 3 classification}, we proved a classification theorem which states for each ordinal number $\Omega$ the proper pure ideals of the commutative ring $$R_{\Omega}=k[X_{\alpha} : \alpha <\Omega]/\big(X_{\alpha}(1-X_{\beta}) : \alpha<\beta<\Omega\big)$$ 
are exactly of the form $I_{\lambda}=(x_{\alpha}: \alpha<\lambda)$ with $\lambda=0$ or $\lambda\leqslant\Omega$ a limit ordinal. To prove this theorem, in addition to the theory of ordinal numbers, a few key lemmas are needed.

The constructions and results developed in this article, especially the ordinal-indexed rings $R_{\Omega}$, make the theory of pure ideals into a whole new world. The above classification theorem, in particular, allows us to compute the pure dimension of the ring $R_{\Omega}$ precisely. Consequently these rings provide explicit examples in which the Krull dimension and pure dimension differ dramatically. In particular, if $\Omega=\omega\cdot n$ with $n\in\omega$, then $\pdim(R_{\Omega})=n$. We also prove that for any nonzero ordinal $\Omega$, the Krull dimension of $R_{\Omega}$ is always 1. Therefore, not only did we prove the existence of a commutative ring having a purely-prime ideal that is not purely-maximal, but we also constructed specific commutative rings $R_{\Omega}$ such that for any natural number $n\geqslant1$, there exists a strict chain of purely-prime ideals of length $n$ inside $R_{\Omega}$. 

The paper also establishes a structural result concerning pseudo-Noetherian rings, that is, rings in which every pure ideal is finitely generated. In fact, in Theorem \ref{Theorem 2}, we completely characterize these rings. In particular, we prove that a ring $R$ is pseudo-Noetherian if and only if $R$ satisfies the ascending chain condition on pure ideals and every purely-prime ideal of $R$ is purely-maximal. We also obtain a precise characterization of all pure ideals, and in particular of all purely-prime ideals, in these rings. As expected, the class of pseudo-Noetherian rings is very broad. In particular, we prove that every ring with finitely many maximal ideals or finitely many minimal prime ideals is pseudo-Noetherian (see Theorem \ref{Theorem 8 finite min-max}). 

Next, we investigate the behavior of pure dimension under fundamental ring constructions. 
In fact, for every commutative ring $R$, 
the pure and purely-prime ideals of the nil-radical quotient $R/N$ and the polynomial ring $R[x]$ are precisely characterized. To prove these results, we use a topological characterization of pure ideals, which describes them precisely as the closed subsets of the prime spectrum that are stable under generalization. The pure and purely-prime ideals of the formal power series ring $R[[x]]$ are also exactly characterized. However, to prove this result, the above topological method does not work, so we proved it using another method (see Theorem \ref{Theorem 6 series ring}).
These results yield the following invariance formula: $$\pdim(R)=\pdim(R/N)=\pdim(R[x])=
\pdim(R[[x]]).$$ 
As an application, we show that every purely-prime ideal of an mp-ring is purely-maximal, extending the result previously established for reduced mp-rings \cite[Theorem 3.5]{Al-Ezeh}.

\section{Characterizing pseudo-Noetherian rings} 

Let $I$ be an ideal of a commutative ring $R$. We say that $I$ is a pure ideal of $R$ if $R/I$ is a flat $R$-module (i.e., the natural ring map $R\rightarrow R/I$ is flat), or equivalently, for each $a \in I$ there exists some $b \in I$ such that $a(1-b)=0$. 

The zero ideal and the whole ring are always pure ideals. 
The sum of any family of pure ideals is a pure ideal. The intersection of finitely many pure ideals is pure.

\begin{remark}\label{Remark 1} If an ideal $I=(a_{i} : i \in S)$ of a ring $R$ is generated by the elements $a_{i}$ such that for each $i$ there exists some $b_i \in I$ with $a_{i}(1-b_{i})=0$ then $I$ is a pure ideal. Indeed, each $a\in I$ can be written as $a=\sum\limits_{d=1}^{m}r_{d}a_{i_{d}}$. Then $b:=1-\prod\limits_{s=1}^{m}(1-b_{i_{s}})\in I$ and $a(1-b)=0$. Hence, $I$ is pure. 
\end{remark}

The extension of a pure ideal under any ring map is a pure ideal. This follows from Remark \ref{Remark 1}. As an alternative proof, it also follows from the fact that a flat ring map is stable under base change.

The only pure ideals of an integral domain are the zero ideal and the whole ring. The same holds for local rings.   

If $I$ is a pure ideal of a ring $R$, then for any ideal $J$ of $R$ we have $IJ=I\cap J$. In particular, every pure ideal is idempotent, i.e., $I=I^2$.

By purely-prime ideal of $R$ we mean a proper pure ideal $I$ of $R$ such that if $J$ and $J'$ are pure ideals of $R$ with $JJ'\subseteq I$ then $J\subseteq I$ or $J'\subseteq I$. Similarly, by purely-maximal ideal of $R$ we mean a proper pure ideal $I$ of $R$ such that it is a maximal element in the set of proper pure ideals of $R$. 

If $I$ is an ideal of a ring $R$, then by $\nu(I)$ we mean the sum of all pure ideals of $R$ contained in $I$. Then $\nu(I)$ is the largest pure ideal of $R$ that is contained in $I$. The ideal $\nu(I)$ is called the pure part of $I$.
If $\mathfrak{p}$ is a prime ideal of $R$, then $\nu(\mathfrak{p})=\nu(\Ker\pi)$ is a purely-prime ideal of $R$ where $\pi:R\rightarrow R_{\mathfrak{p}}$ is the natural ring map. If $\mathfrak{p}\subseteq\mathfrak{q}$ are prime ideals of $R$ then $\nu(\mathfrak{p})=\nu(\mathfrak{q})$. Moreover, every purely-maximal ideal of $R$ is of the form $\nu(\mathfrak{p})$ with $\mathfrak{p}$ a prime ideal of $R$. 

By Zorn’s lemma, it is easy to see that every proper pure ideal is contained in a purely-maximal ideal. In particular, every nonzero ring has at least a purely-maximal ideal. Dually, using Zorn’s lemma together with the pure part of an ideal, one can show that every purely-prime ideal contains a minimal purely-prime ideal. In particular, every nonzero ring has at least a minimal purely-prime ideal.

By Remark \ref{Remark 1}, every ideal that is generated by a set of idempotent elements is a pure ideal. Every finitely generated pure ideal is also generated by an idempotent element. In fact, every finitely generated idempotent ideal is generated by an idempotent element (apply a Nakayama type argument).   

In a ring $R$, by primitive idempotent we mean a nonzero idempotent $e\in R$ such that if $ee'=e'$ for some nonzero idempotent $e'\in R$, then $e=e'$. 

\begin{remark}\label{Remark 2 ACC} It can be seen that a poset $(A,<)$ satisfies the ascending chain condition (ACC) on its elements if and only if every nonempty subset of $A$ has at least a maximal element (with respect to the ordering $<$). Indeed, the direct implication follows from the axiom of choice, and the reverse implication is obvious. 
\end{remark}

By pseudo-Noetherian ring (or semi-Noetherian ring \cite[Definition 6.1]{Tarizadeh}) we mean a ring such that every pure ideal is finitely generated. 

In \cite[Theorem 6.2]{Tarizadeh} we proved a Cohen type theorem which states that if every purely-maximal ideal of a ring $R$ is finitely generated, then $R$ is pseudo-Noetherian. The following theorem establishes a structural bridge between various properties as novel equivalences of pseudo-Noetherian rings, and significantly strengthens  \cite[Theorem 6.2]{Tarizadeh}.

\begin{theorem}\label{Theorem 2} For a ring $R$ the following statements are equivalent: \\
$\mathbf{(i)}$ $R$ is a pseudo-Noetherian ring. \\
$\mathbf{(ii)}$ Every purely-maximal ideal of $R$ is finitely generated.  \\
$\mathbf{(iii)}$ The pure ideals of $R$ are exactly of the form $Re$ with $e \in R$ an idempotent. \\
$\mathbf{(iv)}$ Every purely-prime ideal of $R$ is exactly of the form $R(1-e)$ with $e \in R$ a primitive idempotent. \\
$\mathbf{(v)}$ $R$ satisfies the ascending chain condition on pure ideals and every purely-prime ideal of $R$ is purely-maximal.  \\
In particular, every pseudo-Noetherian ring has finitely many idempotents.  
\end{theorem}

\begin{proof} The implication (i)$\Rightarrow$(ii) is clear, and for the implication (ii)$\Rightarrow$(i) see 
\cite[Theorem 6.2]{Tarizadeh}. \\
(i)$\Rightarrow$(iii): If $e\in R$ is an idempotent, then $Re$ is pure. Conversely, let $I$ be a pure ideal of $R$. By hypothesis, the ideal $I$ is finitely generated.  Since every pure ideal is idempotent, by Nakayama lemma, there exists some $b\in I$ such that $(1-b)I=0$. It follows that $I=Rb$ and $b=b^{2}$. \\
(iii)$\Rightarrow$(iv): Let $I$ be a purely-prime ideal of $R$. Then $I=Rf$ for some idempotent $f\in R$. We show that $e:=1-f$ is a primitive idempotent. Since $I\neq R$, the idempotent $e$ is nonzero. Suppose there exists a nonzero idempotent $e'\in R$ such that $ee'=e'$. Then $e'':=e-e'$ is also an idempotent. We have $e'e''=0$. Then $(Re')(Re'')=Re'e''=0\subseteq R(1-e)=Rf$. Since $I=Rf$ is purely-prime, $Re'\subseteq R(1-e)$ or $Re''\subseteq R(1-e)$. If $Re'\subseteq R(1-e)$ then $e'=e'(1-e)=0$ which is impossible. So $Re''\subseteq R(1-e)$. It follows that $e''=e''(1-e)=0$, so $e=e'$. Therefore, $e$ is a primitive idempotent of $R$. \\ Conversely, let $e$ be a primitive idempotent of $R$. We show that $R(1-e)$ is a purely-prime ideal of $R$. We have $R(1-e)\neq R$, because $e\neq0$. Thus $R(1-e)$ is a proper pure ideal of $R$. Let $I$ and $J$ be pure ideals of $R$ with $IJ\subseteq R(1-e)$. There are idempotents $f, g \in R$ such that $I=Rf$ and $J=Rg$. Then $fg\in R(1-e)$ and so $efg=0$. But $ef$ is an idempotent and $e(ef)=ef$. The primitivity of $e$ implies $ef=0$ or $e=ef$. Likewise, $eg=0$ or $e=eg$. If $e=ef$ then $eg=efg=0$. Thus $ef=0$ or $eg=0$. If $ef=0$ then $f=(1-e)f$ so $Rf\subseteq R(1-e)$. If $eg=0$ then $Rg\subseteq R(1-e)$. Therefore, $R(1-e)$ is a purely-prime ideal of $R$. \\
(iv)$\Rightarrow$(ii): It is clear. \\
(iii)$\Rightarrow$(v): It is clear that every purely-prime ideal of $R$ is purely-maximal. If $I_{1}\subseteq I_{2}\subseteq I_{3}\subseteq\ldots$ is an ascending chain of pure ideals of $R$, then $I:=\bigcup\limits_{n\geqslant1}I_{n}$ is a pure ideal of $R$. So $I=Re$ with $e$ an idempotent. Then $e\in I_{d}$ for some $d\geqslant1$. It follows that $I_{n}=I_{d}=Re$ 
for all $n\geqslant d$. \\
(v)$\Rightarrow$(i): The proof has three steps. \\
Step 1. We first show that the ACC on pure ideals implies that every pure ideal of $R$ is a finite intersection of purely-prime ideals of $R$, allowing the empty intersection to mean $R$. Suppose not. Then $S$, the set of pure ideals of $R$ that do not admit such a representation, is nonempty. 
Now by applying Remark \ref{Remark 2 ACC} to the poset of all pure ideals of $R$ ordered by the (strict) inclusion, we conclude that $S$ has at least a maximal element $I$. 
Then the pure ideal $I$ is proper and is not purely-prime. So there are pure ideals $J$ and $L$ in $R$ with \[JL\subseteq I, \quad J\nsubseteq I, \quad L\nsubseteq I.\]
Both $I+J$ and $I+L$ are pure ideals strictly containing $I$, so each is a finite intersection of purely-prime ideals. However, \[(I+J)\cap(I+L)=(I+J)(I+L)=I^{2}+IL+IJ+JL=I.\]
Hence $I$ is also a finite intersection of purely-prime ideals, a contradiction. \\
Step 2. We show that the ring $R$ can be decomposed as a finite direct product of the rings $R/I$ with $I$ a purely-prime.  The zero ring is immediate, so assume $R\neq0$. Applying Step 1 to the zero ideal and discarding repetitions yields \[\bigcap\limits_{j=1}^{n}I_{j}=
\prod\limits_{j=1}^{n}I_{j}=0\]
where $n\geqslant1$ and the $I_{j}$ are distinct purely-prime ideals of $R$. By hypothesis, they are purely-maximal. Hence, for $i\neq j$, $I_{i}+I_{j}=R$. These are all the purely-prime ideals of $R$. Indeed, if $I$ is a purely-prime ideal of $R$, then $I_{1}\ldots I_{n}=0\subseteq I$ implies $I_{j}\subseteq I$ for some $j$. The pure-maximality of $I_{j}$ gives $I=I_{j}$. Then by the Chinese remainder theorem, the ring map $$\phi:R\rightarrow T=\prod\limits_{j=1}^{n}R/I_{j}$$
$$r\rightarrow(r+I_{1},\ldots, r+I_{n})$$ is an isomorphism. It follows that $I_{j}=Rf_{j}$ where the idempotent $f_{j}\in R$, under this isomorphism, corresponds to $1-e_{j}$ with $e_{j}=(\delta_{i,j})_{i=1}^{n}$ and $\delta_{i,j}$ is the Kronecker delta. Indeed,  $p_{j}\circ\phi: R\rightarrow R/I_{j}$ is the natural ring map and $\phi(f_{j})=1-e_{j}$ where $p_{j}:T\rightarrow R/I_{j}$ is the $j$-th projection map. Then $f_{j}+I_{j}=(p_{j}\circ\phi)(f_{j})=0$. Thus $f_{j}\in I_{j}$, so $Rf_{j}\subseteq I_{j}$. Conversely, if $a\in I_{j}$, then we can write $\phi(a)=\phi(a)(1-e_{j})=\phi(af_{j})$. Then the injectivity,
yields that $a=af_{j}\in Rf_{j}$. \\
Step 3. By Step 1 and the classification of purely-prime ideals in Step 2, every pure ideal $I$ of $R$ has the form \[I=\bigcap\limits_{j\in S}I_{j}=
\prod\limits_{j\in S}I_{j}=Rf\]
for some $S\subseteq\{1,\ldots, n\}$. The element $f=\prod\limits_{j\in S}f_{j}$ is idempotent. When  $S=\emptyset$, the intersection is $R$ and
the product of the $f_{j}$
is $1$. Thus every pure ideal is principal, in particular finitely generated. Therefore
$R$ is pseudo-Noetherian. \\
If the zero ideal of a ring $R$ is purely-prime, then $R$ has no nontrivial idempotents. In particular, if $R$ is pseudo-Noetherian, then the ring decomposition $R\simeq\prod\limits_{j=1}^{n}R/I_{j}$ shows that $R$ has $2^{n}$ idempotents, thus it has finitely many idempotents. \\
As an alternative proof for the last assertion, assume $\mathscr{B}(R)=\{e\in R: e=e^2\}$, the set of idempotents of $R$, is infinite. We know that $\mathscr{B}(R)$ is a Boolean ring whose addition and multiplication are defined respectively as $e\oplus e':=e+e'-2ee'$ and $e\cdot e'=ee'$. Since $\mathscr{B}(R)$ is an infinite Boolean ring, it is well known that there is an infinite set $S=\{e_{1}, e_{2}, e_{3}, \ldots\}$ of nonzero pairwise orthogonal idempotents in $R$, i.e., each $e_{n}\neq0$ and $e_{i}e_{j}=0$ whenever $i\neq j$. Let $I=(S)$ be the ideal of $R$ generated by this set $S$. By Remark \ref{Remark 1}, $I$ is a pure ideal of $R$. Then by  (iii), $I=Re$ with $e\in R$ an idempotent. We may write $e=\sum\limits_{i=1}^{d}r_{i}e_{n_{i}}$. Since $S$ is infinite, we can choose $e_{m}\in S\setminus\{e_{n_{1}},\ldots, e_{n_{d}}\}$. Then $e_{m}\in I=Re$, so $e_{m}=e_{m}e=\sum\limits_{i=1}^{d}r_{i}e_{m}e_{n_{i}}=0$, a contradiction. Hence, $R$ has only finitely many idempotents. 
\end{proof}

By Gelfand ring we mean a ring $R$ such that for each prime ideal $\mathfrak{p}$ of $R$ there exists a unique maximal ideal $\mathfrak{p}'$ of $R$ with $\mathfrak{p}\subseteq\mathfrak{p}'$. Every zero-dimensional ring is clearly a Gelfand ring. Hence, by \cite[Chap. 8, Proposition 37]{Borceux}, every purely-prime ideal of a zero-dimensional ring is purely-maximal. But we provide a simple direct proof:

\begin{lemma}\label{Lemma 1} Every purely-prime ideal of a zero-dimensional ring is purely-maximal.
\end{lemma}

\begin{proof} We first show that every pure ideal $I$ of a zero-dimensional ring $R$ is generated by a set of its idempotent elements. To see this, we use the fact that a ring $R$ is zero-dimensional if and only if for each $a\in R$ there exists some $b\in R$ such that $a^{n}(1-ab)=0$ with $n\geqslant1$. Now take $f\in I$. Since $I$ is pure,  $f=fg$ for some $g\in I$. Then by applying the above fact for the element $g$, there exist a natural number $n\geqslant1$ and some $h\in R$ such that $g^{n}=g^{n+1}h$. It follows that $g^{n}=g^{2n}h^{n}$. This shows that $g^{n}h^{n}$ is idempotent. It is clear that $g^{n}h^{n}\in I$. Then from $f=fg$ we get that $f=fg^{n}=fg^{n}(g^{n}h^{n})$. This shows that $I$ is generated by a set of its idempotent elements. Now the assertion is clear. Indeed, if $I$ is a purely-prime ideal of $R$ then $I\subseteq J$ where $J$ is a purely-maximal ideal of $R$. By the above observation, $J$ is generated by a set of idempotent elements. For any idempotent $e\in J$ we have $Re R(1-e)=0\subseteq I$. Since $I$ is purely-prime, we get that $e\in I$. Therefore, $I=J$. 
\end{proof}

In particular, every Boolean ring and more generally every von-Neumann regular ring is zero-dimensional, and hence every purely-prime ideal is purely-maximal.

\begin{example}\label{Example 3 C(X)} We give an example showing that there exists a ring $R$ with finitely many idempotents and every purely-prime ideal is purely-maximal, but $R$ is not pseudo-Noetherian. For example, take $R=C([0,1])$, the ring of real-valued continuous functions on the unit interval $[0,1]$. It is well known that $R$ is a Gelfand ring, so every purely-prime ideal of $R$ is purely-maximal. In addition, this ring $R$ has no nontrivial idempotents. Let \[I=\{f\in R: \overline{\Supp(f)}\subseteq(0,1)\}\] be the set of all $f\in R$ such that the closure of $\Supp(f)=\{x\in[0,1]:f(x)\neq0\}$ is contained in the open interval $(0,1)$. It can be shown that $I$ is a pure ideal of $R$ which is not finitely generated. Hence, $R$ is not pseudo-Noetherian. \\
In Example \ref{Example 2 nice}, we will see that there are also rings with a finite number of idempotents that satisfy the ACC on pure ideals, but are not pseudo-Noetherian. Therefore, in Theorem \ref{Theorem 2}, neither of the two conditions can be dropped from the equivalent (v). 
\end{example}

Next, we prove the cancellation property of pure ideals modulo the Jacobson radical:

\begin{lemma}\label{Lemma 12 Jacobson} Let $J$ be an ideal of a ring $R$ contained in the Jacobson radical. If $I$ is a pure ideal of $R$ and $L$ is any ideal of $R$ with $I\subseteq J+L$, then $I\subseteq L$. In particular, two pure ideals of $R$ with the same image in $R/J$ are equal. 
\end{lemma}

\begin{proof} If $a\in I$ then $a=ab$ for some $b\in I$. We may write $b=j+l$ with $j\in J$ and $l\in L$. Then $(1-j)a=al$. So $a=(1-j)^{-1}al\in L$. 
\end{proof}

It is clear that every Noetherian ring is pseudo-Noetherian. The same is true for integral domains as well as local rings. In fact, we have a more general result:

\begin{theorem}\label{Theorem 8 finite min-max} Every ring with finitely many maximal ideals or finitely many minimal prime ideals is pseudo-Noetherian. 
\end{theorem}

\begin{proof} The assertion is clear for the zero ring. First assume that the nonzero $R$ has finitely many maximal ideals $\mathfrak{m}_{1},\ldots,\mathfrak{m}_{n}$ with $n\geqslant1$. Let $I$ be a pure ideal of $R$. The finite number of maximal ideals implies that $I+\bigcap\limits_{I\nsubseteq\mathfrak{m}_{i}}
\mathfrak{m}_{i}=R$. Thus there exists $a\in I$ such that $a\notin\bigcup\limits_{I\nsubseteq\mathfrak{m}_{i}}
\mathfrak{m}_{i}$. Since $I$ is pure, $a(1-b)=0$ for some $b\in I$. It follows that $I\subseteq Rb+\mathrm{Jac}(R)$. Indeed, if $x\in I$ then $x(1-b)\in\mathrm{Jac}(R)$. Then by Lemma \ref{Lemma 12 Jacobson}, $I\subseteq Rb$. We also have $Rb\subseteq I$, because $b\in I$. Thus $I=Rb$. Hence, $R$ is pseudo-Noetherian. \\ 
Now assume the nonzero ring $R$ has finitely many minimal prime ideals $\mathfrak{p}_{1},\ldots,\mathfrak{p}_{n}$ with $n\geqslant1$. Let $I$ be a nonzero pure ideal of $R$. By the prime avoidance lemma, $I$ is not contained in $\bigcup\limits_{I\nsubseteq\mathfrak{p}_{i}}
\mathfrak{p}_{i}$. Choose $a\in I$ with $a\notin\bigcup\limits_{I\nsubseteq\mathfrak{p}_{i}}
\mathfrak{p}_{i}$. Since $I$ is pure, $a(1-b)=0$ for some $b\in I$. It follows that $b+N\in R/N$ is idempotent with $N=\bigcap\limits_{i=1}^{n}\mathfrak{p}_{i}$ the nil-radical of $R$. It is well known that the idempotents of every ring can be lifted modulo the nil-radical. So there is an idempotent $e\in R$ with $e-b\in N$. We have $I+N=Re+N$.
Indeed, since $e\in I+N$, so $Re+N\subseteq I+N$. To see the reverse inclusion, take $x\in I$. Then $x(1-e)\in N$.   
This shows that $I+N\subseteq Re+N$. 
Thus by Lemma \ref{Lemma 12 Jacobson}, $I=Re$. Hence, $R$ is pseudo-Noetherian. 
\end{proof}

\begin{remark} Here we provide an alternative proof of Theorem \ref{Theorem 8 finite min-max} that uses more machinery. If a ring $R$ has finitely many minimal prime ideals or finitely many maximal ideals, then it is well known that every finitely generated flat $R$-module is projective (see e.g., \cite[Theorem 3.13]{Tarizadeh 2}). In particular, if $I$ is a pure ideal of $R$ then the $R$-module $R/I$ is projective. It is also well known that for any ring $R$, the annihilator of every finitely generated projective $R$-module is generated by an idempotent element. So $I=Re$ with $e\in R$ an idempotent. Therefore, $R$ is pseudo-Noetherian. 
\end{remark}

The above results, in particular, show that to find a purely-prime ideal that is not purely-maximal, we need to investigate it in the rings having at least a pure ideal that is not finitely generated. On the other hand, if a ring has an infinitely generated pure ideal, then it does not necessarily have the above property (having a purely-prime ideal that is not purely-maximal). 
For instance, in the proof of Lemma \ref{Lemma 1} we observed there are pure ideals which are not necessarily finitely generated. As a specific example, let $k$ be a field and consider the infinite direct product ring $R=\prod\limits_{i\geqslant1}k$. Then the direct sum ideal $I=\bigoplus\limits_{i\geqslant1}k$ of $R$ is a pure ideal which is not finitely generated. However, $R$ is a von-Neumann regular ring, so every purely-prime ideal of $R$ is purely-maximal (see Lemma \ref{Lemma 1}).  

This makes it much more difficult to find a ring with the above property (having a purely-prime ideal that is not purely-maximal). However, the desired example is presented in Theorem \ref{Theorem 1}.  

\section{The first counterexample}

Recall that a family of maps $\{f_i: A\rightarrow B\}_{i\in S}$ is said to separate the elements of $A$ if  $f_{i}(x)=f_{i}(y)$ for all $i$, then $x=y$. Equivalently, for every $x\neq y$ in $A$, there exists some $i\in S$ such that $f_{i}(x)\neq f_{i}(y)$. If the index set $S$ is empty then the separating property holds trivially by vacuous truth. 

We also need a basic fact from polynomial rings which states that if $\{a_{i} : i\in S\}$ is a subset (possibly infinite)  of a ring $R$, then by the universal property, there exists a (unique) surjective morphism of $R$-algebras $\phi:R[X_{i}: i\in S]\rightarrow R$ such that $\phi(X_{i})=a_{i}$ for all $i$. It can be also seen that $\Ker(\phi)=(X_{i}-a_{i}: i\in S)$.

\begin{theorem}\label{Theorem 1} There exists a commutative ring having a purely-prime ideal which is not purely-maximal.
\end{theorem}

\begin{proof} Let $k$ be a field and consider the polynomial ring $k[X_{0},X_{1},X_{2},\ldots]$ modulo the ideal $J$ which is generated by all elements of the form $X_{i}(1-X_{j})$ with $i<j$. We denote this ring by $R_{\omega}:=k[X_{0},X_{1},X_{2},\ldots]/J$ and we write $x_i$ for the image of $X_i$ in $R_{\omega}$ (under the natural ring map). Then in $R_{\omega}$ we have $x_{i}=x_{i}x_{j}$ for all $i<j$. So by Remark \ref{Remark 1}, $I_{\omega}:=(x_{0}, x_{1}, x_{2}, \ldots)$ is a pure ideal of $R_{\omega}$.  \\
We will show that the proper pure ideals of $R_{\omega}$ are exactly the zero ideal and $I_{\omega}$, and consequently, the zero ideal is purely-prime but not purely-maximal. \\
To accomplish this, we require some preparation. By the universal property of polynomial rings, we have a surjective morphism of $k$-algebras $k[X_{0},X_{1},X_{2},\ldots]\rightarrow k$ which maps each $X_{i}$ to $0$. Since, $J$ is contained in the kernel of this map, we get a surjective morphism of $k$-algebras $\epsilon:R_{\omega}\rightarrow k$ such that $\epsilon(x_{i})=0$ for all $i$.  Similarly to the above, for each $n\geqslant0$ we get a surjective morphism of $k$-algebras $\pi_{n}:R_{\omega}\rightarrow k[t]$ such that: 
\[
\pi_{n}(x_{i}) = 
\begin{cases}
    0, & \text{if } i< n, \\
    t,  & \text{if } i=n, \\
    1, & \text{if } i>n.
\end{cases}
\] 
It can be seen that the (ring) maps $\epsilon, \pi_{0}, \pi_{1},\pi_{2},\ldots$ separate the elements of $R_{\omega}$, i.e., if for some $f\in R_{\omega}$ we have $\epsilon(f)=0$ and $\pi_{n}(f)=0$ for all $n$, then $f=0$. 
In fact, since $x_{i}x_{j}=x_{i}$ for all $i<j$, every monomial in $R_{\omega}$ involving several variables reduces to a power of the variable having the smallest index. We denote the image of every $c\in k$, under the natural ring map $k\rightarrow R_{\omega}$, simply by $c$. It follows that every $f\in R_{\omega}$ can be written in the form $f=c+\sum\limits_{i=0}^{N}p_{i}(x_{i})$ where $c\in k$ and each polynomial $p_{i}(t)\in tk[t]$ for all $i$. Then $c=\epsilon(f)=0$. If $f$ is nonzero, then we may choose largest $N$ with $p_{N}(t)\neq0$. We have $0=\pi_{N}(f)=p_{N}(t)$ which is a contradiction. \\
Note that if $f\in R_{\omega}$ then for each $0\leqslant n\leqslant N$ we have $\pi_{n}(f)=c+p_{n}(t)+\sum\limits_{i>n}p_{i}(1)$. For  $n>N$ we have $\pi_{n}(f)=\epsilon(f)=c$. Thus $\pi_{n}(f)(0)=\pi_{n+1}(f)(1)=
c+\sum\limits_{i>n}p_{i}(1)$ for all $n$.\\
Now let $L$ be a nonzero pure ideal of $R_{\omega}$. For each $n$, the image $\pi_{n}(L)$ is a pure ideal of $k[t]$. Since $k[t]$ is an integral domain, its only pure ideals are the zero ideal and $k[t]$, therefore $\pi_{n}(L)\in\{0, k[t]\}$ for all $n$. But these images have the same status, i.e., all images are $0$, or all images are $k[t]$. Indeed, suppose $\pi_{n}(L)= k[t]$ for some $n$. 
Choose $f\in L$ with $\pi_{n}(f)=1$. Then $\pi_{n+1}(f)(1)=\pi_{n}(f)(0)=1$. This shows that $\pi_{n+1}(f)\neq0$ and so $\pi_{n+1}(L)=k[t]$. Similarly, if $n>0$ then $\pi_{n-1}(f)(0)=\pi_{n}(f)(1)=1$, so $\pi_{n-1}(L)=k[t]$. Thus, if one image is nonzero, then all are $k[t]$. But $L\neq0$, so  $\pi_{n}(L)=k[t]$ for all $n$. Indeed, to see this, it suffices to show that $\pi_{n}(L)\neq0$ for some $n$. Suppose $\pi_{n}(L)=0$ for all $n$. We may choose a nonzero $f\in L$. Then by the separating property, $\epsilon(f)\neq0$. But for $n>N$ we have $0=\pi_{n}(f)=\epsilon(f)\neq0$ which is a contradiction. Therefore, $\pi_{n}(L)=k[t]$ for all $n$.  \\
The image $\epsilon(L)$ is a pure ideal of the integral domain $k$, so it is either $0$ or $k$. Assume first that $\epsilon(L)=0$. Then $L\subseteq\Ker(\epsilon)=I_{\omega}$. Take $f\in I_{\omega}$. Then in the presentation of $f$, we have $c=0$ and so $\pi_{n}(f)=0$ for all $n>N$. For each $0\leqslant j \leqslant N$, since $\pi_{j}(L)=k[t]$, we may choose some $g_{j}\in L$ with $\pi_{j}(g_{j})=1$. Then  $g:=1-\prod\limits_{j=0}^{N}(1-g_{j})\in L$ and $\pi_{i}(g)=1$ for all $0\leqslant i \leqslant N$. Therefore $\pi_{n}(f-fg)=0$ for all $n$. Indeed, for $n\leqslant N$, $\pi_{n}(g)=1$, and for $n>N$, $\pi_{n}(f)=0$. Also $\epsilon(f-fg)=0$. Then by the separating property of these maps, we get that $f=fg\in L$. Therefore $L=I_{\omega}$. \\
Now assume that $\epsilon(L)=k$. Choose some $f\in L$ with $\epsilon(f)=1$. We know that $\pi_{n}(f)=\epsilon(f)=1$ for all $n>N$. For each $0\leqslant j \leqslant N$, choose $g_{j}\in L$ with $\pi_{j}(g_{j})=1$. Then $h:=1-(1-f)\prod\limits_{j=0}^{N}(1-g_{j})\in L$ and $\pi_{n}(h)=1$ for all $n$. Also, $\epsilon(h)=1$. The separating property gives $h=1$. Thus $L=R_{\omega}$. 
Therefore, the proper pure ideals of $R_{\omega}$ are exactly $\{0\}$ and $I_{\omega}$. 
Note that the ideal $I_{\omega}$ is nonzero, because $\pi_{0}(x_{0})=t\neq0$. Thus the zero ideal is  purely-prime but not purely-maximal ($I_{\omega}$ is indeed the only purely-maximal ideal of $R_{\omega}$). This completes the proof. 
\end{proof}

In the next section we will extend the ring $R_{\omega}$, constructed in Theorem \ref{Theorem 1}, by indexing its generators by an (infinite) ordinal number. Then, in particular, we show that the new rings have strict chain of purely-prime ideals of any length $n\geqslant1$.

Consider the direct product ring $R=A\times B$. We know that the ideals of $R$ are exactly of the form $I\times J$ where $I$ is an ideal of $A$ and $J$ is an ideal of $B$. It can be seen that the pure ideals of $R$ are exactly of the form $I\times J$ where $I$ is a pure ideal of $A$ and $J$ is a pure ideal of $B$. The purely-prime (resp. purely-maximal) ideals of $R$ are exactly of the form $I\times B$ or of the form $A\times J$ where $I$ is a purely-prime (resp. purely-maximal) ideal of $A$ and $J$ is a purely-prime (resp. purely-maximal) ideal of $B$. The same is true for any finite direct products of rings. 

More generally, consider the (possibly infinite) direct product ring $R=\prod\limits_{i\in S}R_{i}$ where each $R_{i}$ is a nonzero commutative ring. It can be easily seen that an ideal $I$ of $R_{d}$ (for some $d$) is pure if and only if $\pi_{d}^{-1}(I)$ is a pure ideal of $R$ where $\pi_{d}:R\rightarrow R_{d}$ is the natural projection ring map. We have $\pi_{d}^{-1}(I)=\prod\limits_{i\in S}A_{i}$ with $A_{d}=I$ and $A_{i}=R_{i}$ for all $i\neq d$. We also have the following easy fact: 

\begin{proposition} An ideal $I$ of $R_{d}$ is purely-prime if and only if $\pi_{d}^{-1}(I)$ is a purely-prime ideal of $R=\prod\limits_{i\in S}R_{i}$. 
\end{proposition}

\begin{proof} If $I$ is purely-prime, then $\pi_{d}^{-1}(I)$ is a proper pure ideal of $R$. If $J_{1}$ and $J_{2}$ are pure ideals of $R$ with $J_{1}J_{2}\subseteq\pi_{d}^{-1}(I)$, then $\pi_{d}(J_{1})\pi_{d}(J_{2})=\pi_{d}(J_{1}J_{2})\subseteq \pi_{d}\big(\pi_{d}^{-1}(I)\big)=I$. Since $\pi_{d}(J_{1})$ and $\pi_{d}(J_{2})$ are pure ideals, $J_{1}\subseteq\pi_{d}^{-1}(I)$ or $J_{2}\subseteq\pi_{d}^{-1}(I)$. Conversely, if $\pi_{d}^{-1}(I)$ is purely-prime, then $I$ is a proper pure ideal of $R_{d}$. If  $I_{1}$ and $I_{2}$ are pure ideals of $R_{d}$ with $I_{1}I_{2}\subseteq I$, then $\pi_{d}^{-1}(I_{1})\pi_{d}^{-1}(I_{2})\subseteq
\pi_{d}^{-1}(I)$. Since $\pi_{d}^{-1}(I_{1})$ and $\pi_{d}^{-1}(I_{2})$ are pure ideals, $\pi_{d}^{-1}(I_{1})\subseteq\pi_{d}^{-1}(I)$ or  $\pi_{d}^{-1}(I_{2})\subseteq\pi_{d}^{-1}(I)$. It follows that $I_{1}=\pi_{d}\big(\pi_{d}^{-1}(I_{1})\big)\subseteq\pi_{d}
\big(\pi_{d}^{-1}(I)\big)=I$ or $I_{2}=\pi_{d}\big(\pi_{d}^{-1}(I_{2})\big)\subseteq\pi_{d}
\big(\pi_{d}^{-1}(I)\big)=I$. Hence, $I$ is purely-prime.
\end{proof} 

Using the structure of pure ideals in (finite) direct products of rings, we now give two examples, the first is motivated by Theorem \ref{Theorem 2}(iv), and the second by Theorem \ref{Theorem 8 finite min-max}.
  
\begin{example} We give an example of a commutative ring $R$ with a primitive idempotent $e$ such that the pure ideal $R(1-e)$ is not purely-prime.
Consider the pullback $R:=R_{\omega}\times_{k}R_{\omega}$ along the morphism of $k$-algebras $\epsilon:R_{\omega}\rightarrow k$ sending every $x_{i}$ to zero. We know that $R=\{(a,b): \epsilon(a)=\epsilon(b)\}$ is a subring of the direct product ring $R_{\omega}\times R_{\omega}$ containing $J_{1}:=I_{\omega}\times\{0\}$ and $J_{2}:=\{0\}\times I_{\omega}$. Therefore $J_{1}$ and $J_{2}$ are  pure ideals of $R$. The zero ideal of $R_{\omega}$ is purely-prime, so $R_{\omega}$ and hence $R$ have no nontrivial idempotents. Then the unit $e=(1,1)$ is a primitive idempotent of $R$. Thus the pure ideal $R(1-e)=0$ is not purely-prime, because $J_{1}J_{2}=0$ but $J_{1}\neq0$ and $J_{2}\neq0$. As another example, consider the ring given in Example \ref{Example 3 C(X)}.
\end{example} 

\begin{example}\label{Example 4} We give an example of a pseudo-Noetherian ring that has both infinitely many maximal ideals and infinitely many minimal prime ideals, and consequently, the converse of Theorem \ref{Theorem 8 finite min-max} does not hold. Let $k$ be a field and put $A=k[X_{0}, X_{1}, X_{2}, \ldots]/(X_{i}X_{j}: i\neq j)$. We write $x_{i}$ for the image of $X_{i}$ in $A$, and let $\mathfrak{m}=(x_{0}, x_{1}, x_{2},\ldots)$. For each $i$, the ideal $P_{i}=(x_{j}: j\neq i)$ is a minimal prime ideal of $A$. Then the local ring $A_{\mathfrak{m}}$ has infinitely many minimal primes. Now  the pseudo-Noetherian ring $R:=\mathbb{Z}\times A_{\mathfrak{m}}$ has infinitely many maximal ideals and infinitely many minimal primes. 
\end{example}  

\section{Pure dimension and the classification theorem}

As with the notion of Krull dimension, we introduce the notion of pure dimension, by considering chains of purely-prime ideals.  More precisely, for a commutative ring $R$, the Krull dimension $\dim(R)$ is defined as the supremum of the lengths of all finite strict chains of prime ideals in $R$. Note that an infinite strict chain of prime ideals gives the infinite Krull dimension, i.e., $\dim(R)=\infty$. If $R$ is a nonzero commutative ring then it has a maximal ideal, so the Krull dimension $\dim(R)\geqslant0$. Since the supremum of the empty set is $-\infty$, so the Krull dimension of the zero ring is $-\infty$. 

Similarly, the pure dimension of $R$, denoted by $\pdim(R)$, is defined as the supremum of the lengths of all finite strict chains of purely-prime ideals in $R$; an infinite strict chain of purely-primes gives $\pdim(R)=\infty$. If $R$ is a nonzero commutative ring then it has a purely-maximal ideal, so the pure dimension $\pdim(R)\geqslant0$. The pure dimension of the zero ring is $-\infty$ by definition.

Although the Krull dimension and pure dimension are defined analogously, their behavior differs in many cases. For instance, for any ideal $I$ of a ring $R$, we always have $\dim(R/I)\leqslant\dim(R)$, but the corresponding inequality fails for pure dimension:

\begin{example} The pure dimension of the integral domain $A=k[X_{0},X_{1},X_{2},\ldots]$ is zero. However, for the ideal 
$J=\big(X_{i}(1-X_{j}) : i<j\big)$ of $A$ considered in the proof of Theorem \ref{Theorem 1}, we proved that $\pdim(A/J)=1$ (quotienting by an ideal can create new pure ideals and can strictly increase pure dimension). Moreover, the inequality in the opposite direction holds in some cases. For example, the ring $A/J$ modulo the ideal $I_{\omega}=(x_{0}, x_{1}, x_{2}, \ldots)$ is isomorphic to $k$, and hence its pure dimension is zero which is strictly less than $\pdim(A/J)=1$.
\end{example} 

If $\dim(R)=0$, then by Lemma \ref{Lemma 1}, $\pdim(R)=0$.  In fact, a large class of rings, including Noetherian rings, mp-rings (e.g., integral domains), and Gelfand rings (e.g., local rings), have pure dimension zero. So the pure dimension is less subtle than the Krull dimension. For the ring $R_{\omega}$ in Theorem \ref{Theorem 1}, we observed that $\pdim(R_{\omega})=1\leqslant\dim(R_{\omega})$. Therefore, the inequality $\pdim(R)\leqslant\dim(R)$ holds for a broad class of rings $R$. 

However, we will show that there exist commutative rings $R$ for which $\pdim(R)>\dim(R)$. More strikingly, we shall construct rings of Krull dimension 1 whose pure dimensions can be any finite number, or even infinite. To achieve these goals, we need to make some preparations. 

First, we briefly recall ordinal numbers more informally, which is largely sufficient for our purpose here. Recall that by transitive set we mean a set $T$ such that every element $x\in T$ is a subset of $T$ (note that $x$ is itself a set). By ordinal number or simply by ordinal we mean a transitive set such that it is well-ordered with respect to the membership relation $\in$. For any ordinal $\alpha$, its successor $\alpha +1=\alpha^{+}:=\alpha\cup\{\alpha\}$ is also an ordinal. If $\alpha$ and $\beta$ are ordinals, then by $\alpha<\beta$ we mean $\alpha\in\beta$. Every ordinal is indeed the set of all ordinals strictly less than itself.  

The finite ordinals are exactly the natural numbers. More precisely, $0:=\{\}=\emptyset$, $1:=0^{+}=\{0\}$, $2:=1^{+}=\{0,1\}$  and so on. The set of all natural numbers $\omega=\{0,1,2,\ldots\}$  is also an ordinal. Then for each $n\in\omega$, we have a new ordinal $\omega+n$. In particular, 
$\omega+1=\{0,1,2,\ldots, \omega\}$,  $\omega+2=\{0,1,2,\ldots, \omega, \omega+1\}$ and so on. Then we get the ordinal $\omega\cdot2=\omega+\omega=\omega\cup\{\omega+n : n\in\omega\}=\{0,1,2,\ldots, \omega, \omega+1,\omega+2,\ldots\}$. In fact, we get more ordinals $\omega\cdot n$ with $n\geqslant1$. We have then the ordinal $\omega^{2}$. It is well known that the collection of all ordinals is not a set, it is a proper class. 

By limit ordinal we mean a nonzero ordinal that is not the successor of an ordinal. The limit ordinals are $$\omega<\omega\cdot2<\omega\cdot3<
\ldots<\omega^{2}<\ldots$$
If $\alpha<\beta$ are ordinals, then there exists a (unique) nonzero ordinal $\gamma$ such that $\beta=\alpha+\gamma$.     

Now we come to the main setup of this section. Fix a field $k$. For an ordinal (number) $\Omega$, we define $$R_{\Omega}=k[X_{\alpha} : \alpha <\Omega]/\big(X_{\alpha}(1-X_{\beta}) : \alpha<\beta<\Omega\big)$$ and we write $x_{\alpha}$ for the image of $X_{\alpha}$ in $R_{\Omega}$ (under the natural ring map). Thus $x_{\alpha}x_{\beta}=x_{\alpha}$ with $\alpha<\beta$. 

\begin{remark} Under the above notation, the first ordinal-indexed rings are $R_{0}=k$, $R_{1}=k[x_{0}]$, $R_{2}=k[x_{0}, x_{1}]/\big(x_{0}(1-x_{1})\big)$ and so on. For each natural number (finite ordinal) $n$, the ring $R_{n}$ is Noetherian, so $\pdim(R_{n})=0$. The ring in Theorem \ref{Theorem 1} is the case $\Omega=\omega=\{0, 1, 2, \ldots\}$. 
\end{remark} 

By the universal property of polynomial rings, for each $\gamma<\Omega$, there is a surjective morphism of $k$-algebras $$\pi_{\gamma}:R_{\Omega}\rightarrow k[t]$$ such that \[
\pi_{\gamma}(x_{\alpha}) = 
\begin{cases}
    0, & \text{if } \alpha<\gamma, \\
    t,  & \text{if } \alpha=\gamma, \\
    1, & \text{if } \alpha>\gamma.
\end{cases}
\]
In particular, $x_{\alpha}\neq0$ for all $\alpha<\Omega$. Indeed, $\pi_{\alpha}(x_{\alpha})=t\neq0$.

Similarly, for every $\lambda\leqslant\Omega$, there is a surjective morphism of $k$-algebras $$\theta_{\lambda}:R_{\Omega}\rightarrow k$$ such that   
\[
\theta_{\lambda}(x_{\alpha}) = 
\begin{cases}
    0, & \text{if } \alpha<\lambda, \\
    1, & \text{if } \lambda\leqslant\alpha<\Omega.
\end{cases}
\]

If $\lambda=\Omega$, then $\theta_{\lambda}(x_{\alpha})=0$ for all $\alpha$.

Since $x_{\alpha}x_{\beta}=x_{\alpha}$ for all $\alpha<\beta$, every monomial reduces to a power of the variable with least index. Hence, every $f\in R_{\Omega}$ can be written (uniquely) as \[f=c+\sum\limits_{i=1}^{r}p_{i}(x_{\alpha_{i}}), \quad c\in k, \quad \alpha_{1}<\ldots<\alpha_{r}<\Omega, \quad p_{i}(t)\in tk[t]. \qquad (1)\]
The above presentation will be used frequently in the proofs.

In particular, for each $\gamma<\Omega$, we have 
\[\pi_{\gamma}(f)=c+
\sum\limits_{\alpha_{i}>\gamma}p_{i}(1)+\begin{cases}
    p_{j}(t), & \gamma=\alpha_{j} \text{ for some $j$}, \\
    0, & \gamma\notin\{\alpha_{1},\ldots,\alpha_{r}\}.
\end{cases}   \qquad (2)
\]
Moreover, for each $\gamma<\Omega$ with $\gamma+1<\Omega$, one has $\ev_{0}\circ\pi_{\gamma}=\theta_{\gamma+1}=
\ev_{1}\circ\pi_{\gamma+1}$ where $\ev_{0}: k[t]\rightarrow k$ and $\ev_{1}: k[t]\rightarrow k$ are the evaluation maps ($k$-algebra morphisms) at $t=0$ and $t=1$, respectively. In particular,
\[\pi_{\gamma}(f)(0)=
\pi_{\gamma+1}(f)(1), \qquad   (\gamma+1<\Omega). \qquad (3)
\]

To prove the main result of this section (Theorem \ref{Theorem 3 classification}), we need the following key lemmas, which are interesting in their own right.

\begin{lemma}\label{Lemma 2} Let $\Omega$ be a nonzero ordinal. Then the family of maps $(\pi_{\gamma})$ with $\gamma<\Omega$ separate the elements of $R_{\Omega}$.
\end{lemma}

\begin{proof} Assume for some $f\in R_{\Omega}$ we have $\pi_{\gamma}(f)=0$ for all $\gamma<\Omega$. Consider the presentation in (1) for $f$. Then for $\gamma=\alpha_{r}$ we have $\pi_{\gamma}(f)=c+p_{r}(t)=0$ where $p_{r}(t)\in (t)$.  It follows that $c=0$ and $p_{r}(t)=0$. Then by the descending induction, for $\gamma=\alpha_{i}$ with $i<r$ we get that $\pi_{\gamma}(f)=p_{i}(t)=0$ for all $i$. Thus $f=0$. 
\end{proof}

Consider the natural $k$-algebra morphism $\theta_{0}: R_{\Omega}\rightarrow k$ with $\theta_{0}(x_{\alpha})=1$ for all $\alpha<\Omega$. Then the maximal ideal $\Ker(\theta_{0})=(1-x_{\alpha}: 0\leqslant\alpha<\Omega)$ is generated by the single element $1-x_{0}$. Indeed, if $0<\alpha<\Omega$, then $1-x_{\alpha}=(1-x_{\alpha})(1-x_{0})$.

We have then the following vanishing result: 

\begin{lemma}\label{Lemma 3} Let $\Omega$ be an ordinal and consider the natural $k$-algebra morphism $\theta_{0}: R_{\Omega}\rightarrow k$ with $\theta_{0}(x_{\alpha})=1$ for all $\alpha<\Omega$. Then the pure part of $\mathfrak{m}_{\Omega}:=\Ker(\theta_{0})$ is zero. 
\end{lemma}

\begin{proof} The case $\Omega=0$ is clear, so we may assume $\Omega$ is a nonzero ordinal. Let $I$ be a pure ideal of $R_{\Omega}$ with $I\subseteq\Ker(\theta_{0})$. For each $\beta<\Omega$, then $\pi_{\beta}(I)$ is a pure ideal of the integral domain $k[t]$, so $\pi_{\beta}(I)\in\{0, k[t]\}$. We claim that $\pi_{\beta}(I)=0$ for all $\beta$.
Consider the evaluation map ($k$-algebra morphism) $\ev_{1}: k[t]\rightarrow k$ with $t\mapsto1$. Then $\Ker(\ev_{1})=(t-1)$. Since $\theta_{0}=\ev_{1}\circ\pi_{0}$,  we have $\pi_{0}(I)\subseteq (t-1)\neq k[t]$, so $\pi_{0}(I)=0$. \\
Suppose $\pi_{\beta}(I)=k[t]$ for some $\beta$, and we may choose the least such $\beta$, because every nonempty subset of an ordinal has the least element. Then $\beta>0$, and there is $f\in I$ such that $\pi_{\beta}(f)=1$. \\
If $\beta=\gamma+1$ is a successor ordinal, the relation $\pi_{\gamma}(f)(0)=
\pi_{\gamma+1}(f)(1)=\pi_{\beta}(f)(1)=1$ shows that $\pi_{\gamma}(I)\neq0$. So $\pi_{\gamma}(I)=k[t]$, contradicting the minimality of $\beta$. \\
Assume $\beta$ is a limit ordinal. Consider the representation $f=c+\sum\limits_{i=1}^{r}p_{i}
(x_{\alpha_{i}})$ where $c\in k$, $\alpha_{1}<\ldots<\alpha_{r}<\Omega$ and $p_{i}(t)\in tk[t]$ for all $i$. Then in the set of ordinals $\{\alpha_{1}, \ldots,\alpha_{r}\}$,  we may choose
$\gamma < \beta$ larger than every index $< \beta$ occurring in that representative. Then we may write $f=c+\sum\limits_{i=1}^{d}p_{i}
(x_{\alpha_{i}})+\sum\limits_{i=d+1}^{r}p_{i}
(x_{\alpha_{i}})$ where $\alpha_{d}=\gamma$ and $\alpha_{d+1}=\beta$. Then we have $\pi_{\gamma}(f)=c+p_{\alpha_{d}}(t)+
\sum\limits_{i=d+1}^{r}p_{i}(1)$, so $\pi_{\gamma}(f)(0)=c+\sum\limits_{i=d+1}^{r}p_{i}(1)$. Similarly,
$1=\pi_{\beta}(f)=c+p_{\alpha_{d+1}}(t)+
\sum\limits_{i=d+2}^{r}p_{i}(1)$. So $\pi_{\beta}(f)(1)=c+
\sum\limits_{i=d+1}^{r}p_{i}(1)=1$. Therefore, $\pi_{\gamma}(f)(0)=\pi_{\beta}(f)(1)=1$. This shows that $\pi_{\gamma}(I)\neq0$, and we obtain the same contradiction. Therefore, $\pi_{\beta}(I)=0$ for all $\beta<\Omega$. Then by Lemma \ref{Lemma 2}, $I=0$. This completes the proof.
\end{proof}

The above result, in particular, shows that the zero ideal of $R_{\Omega}$ is purely-prime. 

For every $\lambda\leqslant\Omega$, define $I_{\lambda}=(x_{\alpha}: \alpha<\lambda)$ and  $\mathfrak{p}_{\lambda}=\Ker(\theta_{\lambda})$. The ideal $\mathfrak{p}_{\lambda}$ is a maximal ideal of $R_{\Omega}$ and $I_{\lambda}\subseteq \mathfrak{p}_{\lambda}$. In fact, $\mathfrak{p}_{\lambda}=I_{\lambda}+(1-x_{\lambda})$ for $\lambda<\Omega$, and $\mathfrak{p}_{\Omega}=\Ker(\theta_{\Omega})=I_{\Omega}$.

If $\lambda<\mu\leqslant\Omega$ then $I_{\lambda}$ is strictly contained in $I_{\mu}$, i.e., $I_{\lambda}\subset I_{\mu}$. Indeed, $x_{\lambda}\in I_{\mu}$ but $x_{\lambda}\notin I_{\lambda}$. We have then the following result:

\begin{lemma}\label{Lemma 4} Let $\Omega$ be an ordinal. If $\lambda\leqslant\Omega$, then the ring $R_{\Omega}/I_{\lambda}$ is naturally isomorphic to $R_{\Gamma}$ with $\Omega=\lambda+\Gamma$. Moreover, under this isomorphism, the maximal ideal $\mathfrak{p}_{\lambda}/I_{\lambda}$ is mapped onto $\mathfrak{m}_{\Gamma}$ (for $\mathfrak{m}_{\Gamma}$ see Lemma \ref{Lemma 3}). In particular, the pure part of $\mathfrak{p}_{\lambda}/I_{\lambda}$ is the zero ideal. 
\end{lemma}

\begin{proof} In the ring $R_{\Omega}$ we have $$I_{\lambda}=(x_{\alpha}: \alpha<\lambda)=\frac{(X_{\alpha}:\alpha<\lambda)
+J_{\Omega}}
{J_{\Omega}}=\frac{(X_{\alpha}:\alpha<\lambda)+J}
{J_{\Omega}}$$ where $J_{\Omega}=\big(X_{\alpha}(1-X_{\beta}): \alpha<\beta<\Omega\big)$ and $J=\big(X_{\alpha}(1-X_{\beta}): \lambda\leqslant\alpha<\beta<\Omega\big)\subseteq k[X_{\alpha}: \alpha<\Omega]$. Therefore, we have the natural isomorphisms of $k$-algebras: $$R_{\Omega}/I_{\lambda}\simeq\frac{k[X_{\alpha}: \alpha<\Omega]}{(X_{\alpha}:\alpha<\lambda)+J}
\simeq\frac{k[X_{\alpha}: \lambda\leqslant\alpha<\Omega]}{\big(X_{\alpha}(1-X_{\beta}): \lambda\leqslant\alpha<\beta<\Omega\big)}.$$
We also have: $$R_{\Gamma}=k[Z_{\alpha} : \alpha <\Gamma]/\big(Z_{\alpha}(1-Z_{\beta}) : \alpha<\beta<\Gamma\big).$$
Then consider the natural injective morphism of $k$-algebras $$\phi:k[Z_{\alpha} : \alpha <\Gamma]\rightarrow k[X_{\alpha}: \lambda\leqslant\alpha<\Omega]$$ $$Z_{\alpha}\mapsto X_{\lambda+\alpha}$$ for all $\alpha<\Gamma$. This map is also surjective, because if $\lambda\leqslant\alpha<\Omega$ then $\alpha=\lambda+\gamma$ for some $0\leqslant\gamma<\Gamma$. Therefore, the map $\phi$ is an isomorphism. If $\alpha<\beta<\Gamma$ then, by the fundamental property of ordinal addition that is strictly increasing on the right, we have $\lambda+\alpha<\lambda+\beta<\lambda+\Gamma=\Omega$. This shows $\phi(J_{\Gamma})\subseteq J'$
where $J_{\Gamma}:=\big(Z_{\alpha}(1-Z_{\beta}) : \alpha<\beta<\Gamma\big)$ and $J':=\big(X_{\alpha}(1-X_{\beta}): \lambda\leqslant\alpha<\beta<\Omega\big)$. Then we show that $J'\subseteq\phi(J_{\Gamma})$. If $\lambda\leqslant\alpha<\beta<\Omega$ then there exist ordinals $\gamma$ and $\delta$ with $\gamma<\delta<\Gamma$ such that $\alpha=\lambda+\gamma$ and $\beta=\lambda+\delta$. Then $Z_{\gamma}(1-Z_{\delta})\in J_{\Gamma}$ and $X_{\alpha}(1-X_{\beta})=\phi\big(Z_{\gamma}(1-Z_{\delta})
\big)$. Thus we obtain a natural isomorphism of $k$-algebras: $$R_{\Gamma}\rightarrow R_{\Omega}/I_{\lambda}$$  $$z_{\alpha}\mapsto x_{\lambda+\alpha}+I_{\lambda}$$
where $z_{\alpha}$ is the image of $Z_{\alpha}$ in $R_{\Gamma}$. In particular, $1-z_{0}$ is mapped to $1-x_{\lambda}+I_{\lambda}
\in\mathfrak{p}_{\lambda}/I_{\lambda}$. Therefore, under this isomorphism, the maximal $\mathfrak{m}_{\Gamma}=(1-z_{0})$ is mapped onto $\mathfrak{p}_{\lambda}/I_{\lambda}$. Then by Lemma \ref{Lemma 3}, the pure part of $\mathfrak{p}_{\lambda}/I_{\lambda}$ is the zero ideal. This completes the proof. 
\end{proof}

\begin{lemma}\label{Lemma 5} Let $\Omega$ be a nonzero ordinal. If $\lambda\leqslant\Omega$ is a limit ordinal, then $I_{\lambda}=(x_{\alpha}: \alpha<\lambda)$ is the pure part of  $\mathfrak{p}_{\lambda}=\Ker(\theta_{\lambda})$. In particular, $I_{\lambda}$ is a purely-prime ideal of $R_{\Omega}$. 
\end{lemma}

\begin{proof} We first need to show that $I_{\lambda}$ is a pure ideal of $R_{\Omega}$. If $\alpha<\lambda$ then there exists an (ordinal) $\beta$ with $\alpha<\beta<\lambda$, because $\lambda$ is a limit ordinal. Then $x_{\beta}\in I_{\lambda}$ and $x_{\alpha}(1-x_{\beta})=0$. Then by Remark \ref{Remark 1}, $I_{\lambda}$ is a pure ideal. To complete the proof, we need to show that $I_{\lambda}$ is the largest pure ideal of $R_{\Omega}$ which is contained in $\mathfrak{p}_{\lambda}$. Let $J$ be a pure ideal of $R_{\Omega}$ with $J\subseteq\mathfrak{p}_{\lambda}$. Then $(J+I_{\lambda})/I_{\lambda}$ is a pure ideal of $R_{\Omega}/I_{\lambda}$ which is contained in $\mathfrak{p}_{\lambda}/I_{\lambda}$. Then by Lemma \ref{Lemma 4}, $(J+I_{\lambda})/I_{\lambda}$ is the zero ideal. It follows that $J\subseteq I_{\lambda}$. The last part of the assertion is clear, because the pure part of a prime ideal is purely-prime. 
\end{proof}

\begin{lemma}\label{Lemma 6} Let $\Omega$ be a nonzero ordinal, $\beta<\Omega$ and $f\in R_{\Omega}$ with $\pi_{\beta}(f)=1$. If $\beta$ is a limit ordinal, then there exists an ordinal $\delta<\beta$ such that $\pi_{\gamma}(f)=1$ for all $\delta\leqslant\gamma\leqslant\beta$. 
\end{lemma}

\begin{proof} Consider the representation $f=c+\sum\limits_{i=1}^{r}p_{i}(x_{\alpha_{i}})$ with $c\in k$, $\alpha_{1}<\ldots<\alpha_{r}<\Omega$ and $p_{i}(t)\in tk[t]$ for all $i$.  If $\beta=\alpha_{d}$ for some $d$ then $p_{d}(t)=0$, because $1=\pi_{\beta}(f)=c+p_{d}(t)+\sum\limits_{i>d}p_{i}(1)$.
Then take $\delta:=\alpha_{d-1}+1<\beta$ ($\alpha_{d-1}$ is assumed to be $0$ in the case $d=1$). Thus we have $\pi_{\gamma}(f)=c+\sum\limits_{i\geqslant d}p_{i}(1)=
c+\sum\limits_{i>d}p_{i}(1)=1$ for all $\delta\leqslant\gamma\leqslant\beta$.
Now consider the case  $\beta\notin\{\alpha_{1},\ldots,\alpha_{r}\}$. If there are $\alpha_{i}<\beta$, let $\mu$ be the largest one and put $\delta:=\mu+1<\beta$. Otherwise (i.e., $\beta<\alpha_{1}$) put $\delta=0<\beta$. Then $\pi_{\gamma}(f)=\pi_{\beta}(f)=1$ for all $\delta\leqslant\gamma\leqslant\beta$. 
\end{proof}

\begin{remark} Fix $f\in R_{\Omega}$ with the representation $f=c+\sum\limits_{i=1}^{r}p_{i}(x_{\alpha_{i}})$ with $c\in k$, $\alpha_{1}<\ldots<\alpha_{r}<\Omega$ and $p_{i}(t)\in tk[t]$ for all $i$. Then the values $\pi_{\gamma}(f)$ are the same for all $\gamma<\alpha_{1}$. Also, the values $\pi_{\gamma}(f)$ are the same for all $\gamma>\alpha_{r}$. Similarly, for each $i$, the values $\pi_{\gamma}(f)$ are the same for all $\alpha_{i}<\gamma<\alpha_{i+1}$. 
\end{remark}

\begin{lemma}\label{Lemma 9 intersection} Let $\Omega$ be a nonzero ordinal and $\lambda<\Omega$. Then for the ideal $I_{\lambda}=(x_{\alpha}:\alpha<\lambda)$ of $R_{\Omega}$ we have $I_{\lambda}=\bigcap\limits_{\lambda\leqslant
\gamma<\Omega}
\Ker(\pi_{\gamma})$.  
\end{lemma}

\begin{proof} It is clear that $I_{\lambda}\subseteq\bigcap\limits_{\lambda\leqslant
\gamma<\Omega}
\Ker(\pi_{\gamma})$. To see the converse, take $f$ in the intersection, and consider the representation $f=c+\sum\limits_{i=1}^{r}p_{i}(x_{\alpha_{i}})$ with $c\in k$, $\alpha_{1}<\ldots<\alpha_{r}<\Omega$ and $p_{i}(t)\in tk[t]$ for all $i$. If $\alpha_{r}<\lambda$ then $0=\pi_{\lambda}(f)=c$. So $f=\sum\limits_{i=1}^{r}p_{i}(x_{\alpha_{i}})\in I_{\lambda}$. Assume that $\lambda\leqslant\alpha_{r}$. Then $0=\pi_{\alpha_{r}}(f)=c+p_{r}(t)$. It follows that $c=0$ and $p_{r}(t)=0$. In fact, 
$p_{i}(t)=0$ for all $i$ with $\lambda\leqslant\alpha_{i}$. Thus $f=\sum\limits_{\alpha_{i}<\lambda}p_{i}(x_{\alpha_{i}})\in I_{\lambda}$.
\end{proof} 

If $J$ is a pure ideal of $R_{\Omega}$ then for each $\gamma<\Omega$ the ideal $\pi_{\gamma}(J)$ is pure in the domain $k[t]$, so $\pi_{\gamma}(J)\in\{0, k[t]\}$. Then we define $$\mathcal{S}(J)=\{\gamma<\Omega: \pi_{\gamma}(J)=k[t]\}.$$ In this regard, we have the following important observation: 

\begin{lemma}\label{Lemma 7} Let $\Omega$ be an ordinal and let $J$ be a pure ideal of $R_{\Omega}$. Then the set $\mathcal{S}(J)$ is empty (zero ordinal) or $\Omega$ or a limit ordinal. 
\end{lemma}

\begin{proof} We first show that $\mathcal{S}(J)$ is an initial segment of $\Omega$, i.e., if $\alpha\in \mathcal{S}(J)$ and $\gamma<\alpha$, then $\gamma\in \mathcal{S}(J)$. Suppose $\gamma\notin \mathcal{S}(J)$. Let $\beta$ be the least element of the nonempty set $\mathcal{S}(J)\cap(\gamma, \alpha]$. If $\beta=\delta+1$ is a successor, then $\gamma\leqslant\delta$, there exists $f\in J$ with $\pi_{\beta}(f)=1$, then $\pi_{\delta}(f)(0)=\pi_{\delta+1}(f)(1)=1$. This shows that $0\neq\pi_{\delta}(f)\in\pi_{\delta}(J)$, so $\gamma<\delta\in \mathcal{S}(J)$, a contradiction. If $\beta$ is a limit ordinal, then by Lemma \ref{Lemma 6}, there exists an ordinal $\delta<\beta$ such that $[\delta,\beta]\subseteq \mathcal{S}(J)$. But this is in contradiction with the choice of $\beta$ or the assumption that $\gamma\notin \mathcal{S}(J)$. Hence, $\gamma\in\mathcal{S}(J)$. \\
Every initial segment of ordinals is itself an ordinal. Thus there exists an ordinal $\lambda\leqslant\Omega$ such that $\mathcal{S}(J)=\{\gamma: \gamma<\lambda\}=\lambda$. It is clear that $\mathcal{S}(J)$ is empty (zero ordinal) if and only if $J=0$. If $J=R_{\Omega}$ then $\mathcal{S}(J)=\Omega$. 
If $0<\lambda<\Omega$, then $\lambda$ is a limit ordinal.  Indeed, if $\gamma<\Omega$ with $\gamma+1<\Omega$, then $\gamma\in \mathcal{S}(J)$ if and only if $\gamma+1\in \mathcal{S}(J)$. This completes the proof.
\end{proof}

The next lemma is the patching step that upgrades pointwise surjectivity of the maps $\pi_{\gamma}$ to a
single element equal to $1$ on a whole bounded initial interval.

\begin{lemma}\label{Lemma 10 patching} Let $\Omega$ be an ordinal and let $J$ be a pure ideal of $R_{\Omega}$. If $\alpha\in\mathcal{S}(J)$, then there exists $u\in J$ such that $\pi_{\gamma}(u)=1$ for all $\gamma\leqslant\alpha$.
\end{lemma}

\begin{proof} By Lemma \ref{Lemma 7}, every $\gamma\leqslant\alpha$ belongs to $\mathcal{S}(J)$. Let $D$ be the set of ordinals $\delta\leqslant\alpha$ for which there exists $u\in J$ satisfying 
\[
\pi_{\gamma}(u)=1, \qquad (\delta\leqslant\gamma\leqslant\alpha).
\]
The set $D$ is nonempty, because $\alpha\in D$. Let $\delta$ be the least element of $D$, witnessed by $u\in J$. It suffices to show that $\delta=0$. Assume $\delta>0$. If $\delta$ is a limit ordinal, choose $v\in J$ with $\pi_{\delta}(v)=1$. Then Lemma \ref{Lemma 6} gives $\epsilon<\delta$ such that $\pi_{\gamma}(v)=1$ for all $\epsilon\leqslant\gamma\leqslant\delta$. Then $w:=1-(1-u)(1-v)\in J$ and $\pi_{\gamma}(w)=1$ for all $\gamma\in[\epsilon,\alpha]$. Thus $\epsilon\in D$
and $\epsilon <\delta$. This contradicts the minimality of $\delta$. If $\delta=\eta+1$ is a successor ordinal, choose $v\in J$ with $\pi_{\eta}(v)=1$. Then $w:=1-(1-u)(1-v)\in J$ and $\pi_{\gamma}(w)=1$ for all $\gamma\in[\eta,\alpha]$. This shows that $\eta\in D$ and $\eta<\delta$, this again contradicts the minimality of $\delta$. Therefore $\delta=0$.    
\end{proof}

We also need an unbounded version in the endpoint case $S(J)=\Omega$.

\begin{lemma}\label{Lemma 11 tail} Let $\Omega$ be a limit ordinal and let $J$ be a pure ideal of $R_{\Omega}$ such that $\mathcal{S}(J)=\Omega$ and $\theta_{\Omega}(J)=k$. Then $J=R_{\Omega}$. 
\end{lemma}

\begin{proof} Let $E=\{\delta<\Omega :\quad \exists u\in J, \quad\forall\gamma\in[\delta,\Omega), \quad \pi_{\gamma}(u)=1\}$ be the set of all $\delta<\Omega$ such that there exists $u\in J$ with $\pi_{\gamma}(u)=1$ for all $\delta\leqslant\gamma<\Omega$. \\
We first show that the set $E$ is nonempty. Since $\theta_{\Omega}(J)=k$, we may choose $f\in J$ with $\theta_{\Omega}(f)=1$. Consider the representation $f=c+\sum\limits_{i=1}^{r}p_{i}(x_{\alpha_{i}})$ with $c\in k$, $\alpha_{1}<\ldots<\alpha_{r}<\Omega$ and $p_{i}(t)\in tk[t]$ for all $i$. Since $\theta_{\Omega}(x_{\alpha})=0$ for all $\alpha<\Omega$, so $1=\theta_{\Omega}(f)=c$. We may choose an ordinal $\beta$ with $\alpha_{r}<\beta<\Omega$, this is possible because $\Omega$ is a limit ordinal. Then for all $\beta\leqslant\gamma<\Omega$ we have
$\pi_{\gamma}(f)=c=1$. Thus $\beta\in E$. \\
Let $\delta$ be the least element of $E$, witnessed by $u\in J$. Then we show that $\delta=0$. Suppose $\delta>0$. If $\delta$ is a limit ordinal, then Lemma \ref{Lemma 6} gives $\epsilon<\delta$ such that $\pi_{\gamma}(u)=1$ for all $\epsilon\leqslant\gamma\leqslant\delta$. Then $\pi_{\gamma}(u)=1$ for all $\gamma\in[\epsilon,\Omega)$. This shows that $\epsilon \in E$.
This contradicts the minimality of $\delta$. If $\delta=\eta+1$ is a successor ordinal, choose $v\in J$ with $\pi_{\eta}(v)=1$. This is possible because 
by hypothesis, $\eta\in\Omega=\mathcal{S}(J)$. Then $w:=1-(1-u)(1-v)\in J$ and $\pi_{\gamma}(w)=1$ for all $\gamma\in[\eta,\Omega)$. This shows that $\eta\in E$, this again contradicts the minimality of $\delta$. Therefore $\delta=0$. Thus $\pi_{\gamma}(u)=1$ for all $\gamma<\Omega$. Then the separating property gives that $1=u\in J$. So $J=R_{\Omega}$. 
\end{proof}

We come now to the following classification result which is the culminating point of the above constructions.

\begin{theorem}\label{Theorem 3 classification} Let $\Omega$ be an ordinal. Then the proper pure ideals of the ring $R_{\Omega}$ are exactly of the form $I_{\lambda}=(x_{\alpha}: \alpha<\lambda)$ with $\lambda=0$ or $\lambda\leqslant\Omega$ a limit ordinal. Moreover, all of them are purely-prime. 
\end{theorem}

\begin{proof} If $\lambda\leqslant\Omega$ is a limit ordinal, then by Lemma \ref{Lemma 5}, $I_{\lambda}$ is a proper pure ideal of $R_{\Omega}$. Conversely, let $J$ be a proper pure ideal of $R_{\Omega}$. Thus by Lemma \ref{Lemma 7}, $\mathcal{S}(J)=\lambda\leqslant\Omega$ is an ordinal. Then we have two cases. \\
Case 1: $\lambda<\Omega$. Then for every $\lambda\leqslant\gamma<\Omega$, we have $\pi_{\gamma}(J)=0$. Then by Lemma \ref{Lemma 9 intersection}, $J\subseteq I_{\lambda}$. If $\lambda=0$, this already gives $J=0$. Otherwise, $\lambda$ is a limit ordinal (see Lemma \ref{Lemma 7}). \\
We show that $I_{\lambda}\subseteq J$. If $\alpha<\lambda$, then $\alpha\in \mathcal{S}(J)$. Then by Lemma \ref{Lemma 10 patching}, there exists $u\in J$ such that $\pi_{\gamma}(u)=1$ for all $\gamma\leqslant\alpha$. Therefore $\pi_{\gamma}(x_{\alpha}(1-u))=0$ for all $\gamma<\Omega$. Then the separating property gives $x_{\alpha}=x_{\alpha}u\in J$. Hence $I_{\lambda}\subseteq J$. So $J=I_{\lambda}$. \\
Case 2: $\lambda=\Omega$. If $\Omega=\rho+1$ is a successor ordinal, Lemma \ref{Lemma 10 patching} applied to $\alpha=\rho$ gives $u\in J$ with $\pi_{\gamma}(u)=1$ for all $\gamma<\Omega$. Separation gives $1=u\in J$, contradicting the properness of $J$.  Thus this
case can occur for a proper ideal only when $\Omega$ is a limit ordinal. \\
Assume now that $\Omega$ is a limit ordinal. The ideal $\theta_{\Omega}(J)$ is a pure ideal of the field, so it is $0$ or $k$. The latter possibility $\theta_{\Omega}(J)=k$ would imply $J=R_{\Omega}$ by Lemma \ref{Lemma 11 tail}. This is in contradiction with the properness of $J$. Hence $\theta_{\Omega}(J)=0$. Thus $J\subseteq \Ker(\theta_{\Omega})=I_{\Omega}$. \\
For the reverse inclusion, take $\alpha<\Omega$. Then by Lemma \ref{Lemma 10 patching}, there exists $u\in J$ such that $\pi_{\gamma}(u)=1$ for all $\gamma\leqslant\alpha$. Thus $\pi_{\gamma}(x_{\alpha}(1-u))=0$ for all $\gamma<\Omega$. Separation gives $x_{\alpha}=x_{\alpha}u\in J$. Therefore $J=I_{\Omega}$. This completes the proof. 
\end{proof}

As a first application, Theorem \ref{Theorem 1} follows immediately: 

\begin{corollary} The proper pure ideals of the ring $R_{\omega}=k[X_{0}, X_{1}, X_{2}, \ldots]/\big(X_{i}(1-X_{j}) : i<j\big)$ are exactly the zero ideal and $I_{\omega}=(x_{0}, x_{1}, x_{2}, \ldots)$. In particular, there exists a commutative ring having a purely-prime ideal that is not purely-maximal. 
\end{corollary}

\begin{proof} It follows from Theorem \ref{Theorem 3 classification}.
\end{proof}

\begin{lemma}\label{Lemma 8 Krull 1} If $\Omega$ is a nonzero ordinal, then $\dim(R_{\Omega})=1$.
\end{lemma}

\begin{proof} We first show that $\dim(R_{\Omega})\leqslant1$. To see this, it suffices to show that if $\mathfrak{p}$ is a prime ideal of $R_{\Omega}$ then $\dim(R_{\Omega}/\mathfrak{p})\leqslant1$.  We have $\mathfrak{p}=P/J$ where $P$ is a prime ideal of $A=k[X_{\alpha} : \alpha <\Omega]$ containing $J=\big(X_{\alpha}(1-X_{\beta}) : \alpha<\beta<\Omega\big)$. Consider the disjoint sets $L=\{\alpha\in\Omega: x_{\alpha}\in\mathfrak{p}\}$ and $U=\{\alpha\in\Omega: 1-x_{\alpha}\in\mathfrak{p}\}$. If $L\cup U=\Omega$, then $\big(X_{\alpha} : \alpha \in L\big)+\big(1-X_{\alpha} : \alpha \in U\big)$ is a maximal ideal of $A$ which is contained in $P$. Hence $P$ is a maximal ideal of $A$. Then $\dim(R_{\Omega}/\mathfrak{p})=0$. \\ 
If $L\cup U\neq\Omega$, then there is a unique $\gamma <\Omega$ such that $\gamma\notin L\cup U$. Then consider the morphism of $k$-algebras $k[t]\rightarrow R_{\Omega}/\mathfrak{p}$ which maps $t$ into $x_{\gamma}+\mathfrak{p}$. But the ring $R_{\Omega}/\mathfrak{p}$ as a $k$-algebra is generated by the single element $x_{\gamma}+\mathfrak{p}$. Indeed for $\alpha<\gamma$, the relation $x_{\alpha}(1-x_{\gamma})=0$ forces $x_{\alpha}\in\mathfrak{p}$, for $\gamma<\beta$, the relation $x_{\gamma}(1-x_{\beta})=0$ forces $1-x_{\beta}\in\mathfrak{p}$. So the above map is surjective. Thus, $\dim(R_{\Omega}/\mathfrak{p})\leqslant\dim(k[t])=1$. Therefore, $\dim(R_{\Omega})\leqslant1$. \\
On the other hand, for any $\gamma<\Omega$, the prime ideal $\Ker(\pi_{\gamma})=\pi_{\gamma}^{-1}(0)$ is strictly contained in the prime ideal $\pi_{\gamma}^{-1}((t))$, because $\pi_{\gamma}(x_{\gamma})=t\neq0$. Since $\Omega$ is nonzero, so there exists at least one $\gamma<\Omega$.
Hence, $\dim(R_{\Omega})\geqslant1$. 
\end{proof} 

The division theorem for ordinals states that for any ordinal $\alpha$ and any nonzero ordinal $\beta$, then there are unique ordinals $\gamma$ and $\delta$ such that $$\alpha=\beta\cdot\gamma+\delta, \qquad \delta<\beta.$$ 

So every ordinal $\Omega$ is written uniquely as $$\Omega=\omega\cdot\alpha+n, \qquad n\in\omega.$$
In particular, $\Omega$ is a limit ordinal if and only if $\Omega=\omega\cdot\alpha$ with $\alpha>0$ ($n=0$). Then $\{\omega\cdot\beta : 1\leqslant\beta\leqslant\alpha\}$ is the set of all limit ordinals up to $\Omega$. Every $\Omega<\omega^{2}$ is written uniquely as $\Omega=\omega\cdot n+d$ with $n,d\in\omega$.

\begin{corollary} For an ordinal $\Omega=\omega\cdot\alpha+n$ with $n\in\omega$ 
we have: \\
\[\pdim(R_{\Omega})= 
\begin{cases}
    \alpha & \text{if } \alpha<\omega, \\
    \infty & \text{if } \alpha\geqslant\omega.
\end{cases}
\]
In particular, if $\Omega\geqslant\omega\cdot2$, then $\pdim(R_{\Omega})\geqslant2>1=\dim(R_{\Omega})$. 
\end{corollary}

\begin{proof} The first part of the assertion is now clear, because by the classification theorem the ring $R_{\Omega}$ has a unique maximal strict chain of purely‑prime ideals
$$I_{0}=0\subset I_{\omega}\subset I_{\omega\cdot2}\subset I_{\omega\cdot 3}\subset\ldots\subset I_{\omega\cdot\alpha}.$$
If $\Omega\geqslant\omega\cdot2$ then also apply Lemma \ref{Lemma 8 Krull 1}.
\end{proof}

\begin{corollary} If $\Gamma<\Omega$ are ordinals, then $\pdim(R_{\Gamma})\leqslant\pdim(R_{\Omega})$. 
\end{corollary}

\begin{proof} If $\Gamma<\Omega$ then the set of limit ordinals at most $\Gamma$ is contained in the set of limit ordinals at most $\Omega$. Then apply the above corollary. 
\end{proof}

\begin{lemma}\label{Lemma 13 connected} If $\Omega$ is an ordinal, then the ring $R_{\Omega}$ has no nontrivial idempotents, i.e., $\Spec(R_{\Omega})$ is connected.
\end{lemma}

\begin{proof} If $\Omega=0$ then the field $R_{\Omega}=k$ has no nontrivial idempotents. Assume $\Omega$ is a nonzero ordinal. If $f\in R_{\Omega}$ is an idempotent then $\pi_{\gamma}(f)$ is an idempotent of the domain $k[t]$, so  $\pi_{\gamma}(f)\in\{0,1\}$ for all $\gamma<\Omega$. 
If $\pi_{0}(f)=0$ then we show that $\pi_{\gamma}(f)=0$ for every $\gamma<\Omega$.
Suppose $\pi_{\beta}(f)=1$ for some $\beta<\Omega$. Then $0<\beta$.
Choose $\beta$ minimal with $\pi_{\beta}(f)=1$. If $\beta=\gamma+1$ is successor, then the evaluation identity (3) gives $\pi_{\gamma}(f)(0)=\pi_{\gamma+1}(f)(1)=1$. Since $\pi_{\gamma}(f)\in\{0,1\}$, so $\pi_{\gamma}(f)=1$, contradicting the minimality of $\beta$. If $\beta$ is a limit ordinal, write $f=c+\sum\limits_{i=1}^{r}p_{i}(x_{\alpha_{i}})$ with $c\in k$, $\alpha_{1}<\ldots<\alpha_{r}<\Omega$ and $p_{i}(t)\in tk[t]$ for all $i$. Then we may choose $\gamma<\beta$ larger than every $\alpha_{i}<\beta$. Thus, for every $i$, if $\alpha_{i}<\beta$ then $\alpha_{i}<\gamma$. We have $\pi_{\gamma}(f)=c+\sum\limits_{\alpha_{i}>\gamma}p_{i}(1)$
and 
\[\pi_{\beta}(f)=c+
\sum\limits_{\alpha_{i}>\beta}p_{i}(1)+\begin{cases}
    p_{j}(t), & \beta=\alpha_{j} \text{ for some $j$}, \\
    0, & \beta\notin\{\alpha_{1},\ldots,\alpha_{r}\}.
\end{cases}  
\]
Thus $\pi_{\gamma}(f)=\pi_{\beta}(f)(1)=1$, again contradicting the minimality of $\beta$. Therefore, $\pi_{\gamma}(f)=0$ for every $\gamma<\Omega$. Then the separating property gives that $f=0$. \\
If $\pi_{0}(f)=1$, we apply the preceding argument to the idempotent $1-f$. Since $\pi_{0}(1-f)=0$, we obtain $\pi_{\gamma}(1-f)=0$ for every $\gamma<\Omega$. Then separating gives that $f=1$. \\
As an alternative proof, by Lemma \ref{Lemma 3}, the zero ideal of $R_{\Omega}$ is purely-prime, and hence this ring has no nontrivial idempotents. 
\end{proof}

We obtain a ring-theoretic characterization for limit ordinals:

\begin{corollary} Let $\Omega$ be an ordinal and $\lambda\leqslant\Omega$. Then $\lambda$ is a limit ordinal if and only if the ideal $I_{\lambda}=(x_{\alpha} : \alpha<\lambda)\subseteq R_{\Omega}$ is not finitely generated.
\end{corollary}

\begin{proof} If $\lambda$ is a limit ordinal, then by Lemma \ref{Lemma 5}, $I_{\lambda}$ is a proper pure ideal of $R_{\Omega}$. If $I_{\lambda}$ is finitely generated, 
then $I_{\lambda}$ is generated by an idempotent $e\in R_{\Omega}$. Since $I_{\lambda}$ is proper, $e\neq1$. If $e=0$ then $x_{\alpha}\in I_{\lambda}=0$ for all $\alpha<\lambda$. Thus $e\notin\{0,1\}$. But this is in contradiction with Lemma \ref{Lemma 13 connected}. Therefore, the ideal $I_{\lambda}$ is not finitely generated. Conversely, if $\lambda=0$ or $\lambda=\beta+1$ is a successor ordinal, then $I_{\lambda}=0$ or $I_{\lambda}=(x_{\alpha} : \alpha<\lambda)=(x_{\beta})$, respectively. 
Thus $\lambda$ is a limit ordinal. 
\end{proof}

\begin{example}\label{Example 2 nice} If $R$ is a pseudo-Noetherian ring, then we know that $R$ satisfies the ACC on pure ideals. However, unlike for Noetherian rings, the converse is not true. For example, we proved that the proper pure ideals of $R_{\omega}$ are exactly $0 \subset I_{\omega}$. So $R_{\omega}$ satisfies the ACC on pure ideals. But $R_{\omega}$ is not pseudo-Noetherian, because the pure ideal $I_{\omega}$ is not finitely generated. 
\end{example}

\section{Invariance of pure dimension}

To prove the main results of this section, we need the following known lemma which provides a topological characterization for pure ideals: 

\begin{lemma}\label{Lemma 14 top} For any ring $R$, the map $I\mapsto V(I)$ is a bijection from the set of pure ideals of $R$ onto the set of Zariski closed subsets of $\Spec(R)$ that are stable under generalization. 
\end{lemma}

\begin{proof} To see injectivity, assume $V(I)=V(J)$ for some pure ideals $I$ and $J$ of $R$. Then $\sqrt{I}=\sqrt{J}$. If $a\in I$ then $a=ab$ for some $b\in I\subseteq\sqrt{I}=\sqrt{J}$. Thus $b^{n}\in J$ for some $n\geqslant1$. Then, by iterating, we get that $a=ab^{n}\in J$. So $I\subseteq J$. Similarly, we get that $J\subseteq I$. Therefore, $I=J$. To see surjectivity, take a closed subset $V(I)$ that is stable under generalization where $I$ is an ideal of $R$.
Let $J$ be the kernel of the natural ring map $R\rightarrow S^{-1}R$ with $S=1+I$. Then it is proved that $V(I)=V(J)$ and $J$ is a pure ideal of $R$. For its proof see e.g., \cite[Theorem 3.1]{Tarizadeh 2}. It can be seen that $J$ is indeed the pure part of $I$.
\end{proof}

Recall that a poset $(A,<)$ is said to admit binary joins if for all $x,y\in A$ the join $x\vee y:=\sup\{x,y\}$ exists in 
$A$. Assume that a poset $A$ admits binary joins, then by join-prime element of $A$ we mean an element $p\in A$ such that if $p\leqslant x\vee y$  for some $x,y \in A$, then $p\leqslant x$ or $p\leqslant y$.

If $f:X\rightarrow Y$ is a continuous map of topological spaces and a subset $E\subseteq Y$ is stable under generalization, then $f^{-1}(E)$ is stable under generalization. By $\mathscr{L}(X)$ we mean the set of nonempty closed subsets of $X$ that are stable under generalization. The poset $\mathscr{L}(X)$ ordered by inclusion admits binary joins, i.e., if $E,F\in \mathscr{L}(X)$ then $E\vee F:=\sup\{E, F\}=E\cup F\in \mathscr{L}(X)$. \\
If $f$ is also an onto map, then it induces a morphism of posets $\mathscr{L}(Y)\rightarrow\mathscr{L}(X)$ given by $E\mapsto f^{-1}(E)$, i.e., it is an order-preserving map. Note that a given morphism of posets, in general, does not necessarily preserve binary joins. 

For Krull dimension, we always have $\dim(R)=\dim(R/N)$ with $N$  the nil-radical of $R$. We prove that the same equality holds for pure dimension, but its proof is more subtle than that for Krull dimension:

\begin{theorem}\label{Theorem 4 unexpected} Let $R$ be a ring with the nil-radical $N$. Then the pure ideals of $R/N$ are exactly of the form $(I+N)/N$ with $I$ a pure ideal
of $R$. The same is true for purely-primes. In particular, $\pdim(R)=\pdim(R/N)$.
\end{theorem}

\begin{proof} 
Let $\mathscr{L}(R)$ be the set of nonempty closed subsets of $\Spec(R)$ that are stable under generalization. The poset $\mathscr{L}(R)$ ordered by inclusion admits binary joins. By Lemma \ref{Lemma 14 top}, $\mathscr{L}(R)$ is the set of all $V(I)$ where $I$ is a proper pure ideal of $R$. It is clear that the map $I\rightarrow V(I)$ is order reversing, and under this map, the set of purely-prime ideals of $R$ correspond to the join-prime elements of the poset $\mathscr{L}(R)$. \\
On the other hand, the natural ring map $R\rightarrow R/N$ induces a homeomorphism between the corresponding prime spectra $\phi:\Spec(R/N)\rightarrow\Spec(R)$. This homeomorphism induces an isomorphism between the corresponding posets $\mathscr{L}(R)\rightarrow\mathscr{L}(R/N)$ that is given by $E=V(I)\mapsto\phi^{-1}(E)=V((I+N)/N)$. Every isomorphism of posets preserves binary joins and join-prime elements. So a given $E\in \mathscr{L}(R)$ is join-prime if and only if 
$\phi^{-1}(E)$ is join-prime. \\  
Then, by combining these maps, we get an order preserving bijection from the set of pure ideals of $R$ onto the set of pure ideals of $R/N$ that is given by $I\mapsto(I+N)/N$. Under this map, the purely-prime ideals of $R$ are in bijection with the purely-prime ideals of $R/N$. In particular, $\pdim(R)=\pdim(R/N)$.
\end{proof}

\begin{remark}  The above result is supported by the following observation. If $J$ is an ideal of a commutative ring $R$ contained in the Jacobson radical, then by Lemma \ref{Lemma 12 Jacobson}, the map $I\mapsto(I+J)/J$ is injective from the set of pure ideals of $R$ into the set of pure ideals of $R/J$. However, this map is not surjective in general. Moreover, if $I$ is purely-prime then its extension $(I+J)/J$ is a proper pure ideal of $R/J$, but it is not necessarily purely-prime. For example, let $p$ and $q$ be distinct prime numbers and put $S=\mathbb{Z}\setminus(p\mathbb{Z}\cup q\mathbb{Z})$. Then the ring $R=S^{-1}\mathbb{Z}$ is semilocal with the maximal ideals $pR$ and $qR$. By the Chinese remainder theorem, $R/J\simeq\mathbb{F}_{p}\times\mathbb{F}_{q}$ where $J=pqR$ is the Jacobson radical of $R$. Since $R$ is a domain, its pure ideals are trivial. But $R/J$ has nontrivial pure ideals $\mathbb{F}_{p}\times\{0\}$ and $\{0\}\times\mathbb{F}_{q}$. Hence, the above map is not surjective. The zero ideal of $R$ is purely-prime, but its extension in $R/J$ is not purely-prime. \\
In particular, the map $I\mapsto(I+N)/N$, defined from the set of pure ideals of $R$ into the set of pure ideals of $R/N$, is injective. As an alternative proof, if $I_{1}+N=I_{2}+N$ for some pure ideals $I_{1}$ and $I_{2}$ of $R$, then $\sqrt{I_{1}+N}=\sqrt{I_{2}+N}$. Since $N=\sqrt{0}$, we obtain that $\sqrt{I_{1}}=\sqrt{I_{2}}$. So $I_{1}=I_{2}$, because they are pure. However, the proof of its surjectivity is not easy and requires Theorem \ref{Theorem 4 unexpected}. 
\end{remark}

By mp-ring we mean a ring $R$ such that for each prime ideal $\mathfrak{p}$ of $R$ there exists a unique minimal prime ideal $\mathfrak{p}'$ of $R$ with $\mathfrak{p}'\subseteq\mathfrak{p}$. In \cite[Theorem 3.5]{Al-Ezeh}, it was proved that every purely-prime ideal of a reduced mp-ring is purely-maximal. This result is now strengthened: 

\begin{corollary}\label{Coro 2 mp} Every purely-prime ideal of an mp-ring is purely-maximal. 
\end{corollary}

\begin{proof} If $R$ is an mp-ring, then $R/N$ is a reduced mp-ring, so every purely-prime ideal of $R/N$ is purely-maximal. Then by Theorem \ref{Theorem 4 unexpected}, we have $\pdim(R)=\pdim(R/N)=0$.
\end{proof}

For any ordinal $\Omega$, by the separating property, the ring $R_{\Omega}$ is reduced. If $\Omega\geqslant\omega$, then by Theorem \ref{Theorem 3 classification} and Corollary \ref{Coro 2 mp}, $R_{\Omega}$ is not an mp-ring. 

The radical of a purely-prime ideal is not necessarily a prime ideal. For example, the zero ideal of the reduced $R_{\Omega}$ is purely-prime. If $\Omega>0$, then the zero ideal is not a prime ideal.  

\begin{remark} The above topological characterization of pure and purely-prime ideals allows us to naturally generalize the pure dimension notion to any topological space, especially to any scheme (not necessarily affine). But in this article we focused only on the affine case. 
\end{remark}

Recall the easy fact that if $\phi: A\rightarrow B$ is a ring map, then for any ideals $I$ and $J$ of $A$ we have $(IJ)^{e}=I^{e}J^{e}$ where by $I^{e}=(\phi(a) : a\in I)$ we mean the extension ideal of $I$ under $\phi$. 

\begin{remark} Let $I$ be an ideal of a commutative ring $R$. By $I[x]$ we mean the set of all polynomials $f=a_0 + a_1 x+...+a_n x^n \in R[x]$ such that all coefficients of $f$ are members of $I$. Then the set $I[x]$ is an ideal of $R[x]$. In fact $I[x]$ is the extension of $I$ under the natural ring map $R\rightarrow R[x]$. In addition, $I\subseteq J$ if and only if $I[x]\subseteq J[x]$. Similarly, by $I[[x]]$ we mean the set of all formal power series $f(x)=\sum\limits_{i\geqslant0} a_i x^i \in R[[x]]$ such that $a_i \in I$ for all $i$. The set $I[[x]]$ is also an ideal of $R[[x]]$. But the ideal $I^{e}$, the extension of $I$ under the natural ring map $R\rightarrow R[[x]]$, is contained in the coefficientwise ideal  $I[[x]]$. This inclusion can be strict (see Example \ref{Example 1 coefficientwise}). In fact, $$I^{e}=I(R[[x]])\subseteq (IR)[[x]]=I[[x]].$$
Note that using $IR[[x]]$ (without parentheses) to denote $I^{e}$ is quite confusing and leads to mistakes. Indeed, since $IR=I$, so $IR[[x]]=I[[x]]$. But, in general, $I^{e}\neq I[[x]]$. 
\end{remark}

\begin{theorem}\label{Theorem 5 polynomial} For every commutative ring $R$, the pure ideals of the polynomial ring $R[x]$ are exactly of the form $I[x]$ with $I$ a pure ideal of $R$. The same is true for purely-primes. In particular, $\pdim(R)=\pdim(R[x])$.
\end{theorem}

\begin{proof} If $I$ is a pure ideal of $R$, then $I[x]$ is a pure ideal of $R[x]$, because the extension of a pure ideal is pure. Now let $J$ be a pure ideal of $R[x]$. Then the closed subset $V(J)\subseteq\Spec(R[x])$ is stable under generalization. The map $s:\Spec(R)\rightarrow\Spec(R[x])$  given by $\mathfrak{p}\mapsto\mathfrak{p}[x]$ is continuous. Indeed, if $f=\sum\limits_{i=0}^{n}a_{i}x^{i}\in R[x]$ then $s^{-1}(D(f))=\bigcup\limits_{i=0}^{n}D(a_{i})$. So $E:=s^{-1}(V(J))$ is a closed subset of $\Spec(R)$. The set $E$ is also stable under generalization. Indeed, if $\mathfrak{p}\in E$ and $\mathfrak{q}\subseteq\mathfrak{p}$, then $\mathfrak{q}[x]\subseteq\mathfrak{p}[x]=s(\mathfrak{p})\in V(J)$. Since $V(J)$ is stable under generalization, $s(\mathfrak{q})=\mathfrak{q}[x]\in V(J)$. Then $\mathfrak{q}\in s^{-1}(V(J))=E$. Thus by Lemma \ref{Lemma 14 top}, there is a (unique) pure ideal $I$ in $R$ such that $E=V(I)$. We show that $J=I[x]$. To see this, it suffices to show that $V(J)=V(I[x])$ in $\Spec(R[x])$. Let $P$ be a prime ideal of $R[x]$ with $\mathfrak{p}:=P\cap R$. If $P\in V(I[x])$ then $I\subseteq\mathfrak{p}$. So $\mathfrak{p}\in V(I)=s^{-1}(V(J))$. Thus $s(\mathfrak{p})=\mathfrak{p}[x]\in V(J)$. Then $J\subseteq\mathfrak{p}[x]\subseteq P$. Thus $P\in V(J)$. This shows that $V(I[x])\subseteq V(J)$. For the reverse inclusion, take $P\in V(J)$. Then $\mathfrak{p}[x]\subseteq P$. Thus $s(\mathfrak{p})=\mathfrak{p}[x]\in V(J)$, because $V(J)$ is stable under generalization. So $\mathfrak{p}\in s^{-1}(V(J))=V(I)$. Thus $I\subseteq\mathfrak{p}$. Hence, $I[x]\subseteq\mathfrak{p}[x]\subseteq P$. Therefore, $P\in V(I[x])$. \\
Now, it is easy to see that an ideal $I$ of $R$ is purely-prime (resp. purely-maximal) if and only if $I[x]$ is a purely-prime (resp. purely-maximal) ideal of $R[x]$. In particular, $\pdim(R)=\pdim(R[x])$. 
\end{proof}

\begin{remark} By the same argument, Theorem \ref{Theorem 5 polynomial} holds for polynomial rings in any number of variables (possibly infinitely many). More generally, it holds for any monoid-ring $R[M]$, provided that $M$ is a commutative cancellative torsion-free monoid. More precisely, in this case, the pure ideals of $R[M]$ are exactly of the form $I[M]$ with $I$ a pure ideal of $R$ (the same is true for purely-primes). In particular, $\pdim(R)=\pdim(R[M])$. However, the theorem fails for general monoids. For example, consider the commutative monoid $M$ with presentation \[M=\langle m_{0}, m_{1},m_{2},\ldots | m_{i}m_{j}=m_{i} \quad i<j\rangle.\]
Then by the definition of monoid-rings, $k[M]$ is naturally isomorphic to $R_{\omega}$.  Thus $\pdim(k[M])=\pdim(R_{\omega})=1$. But $\pdim(k)=0$.
\end{remark}

As an application of Theorem \ref{Theorem 5 polynomial}, we obtain a Hilbert type basis theorem:

\begin{corollary} If $R$ is a pseudo-Noetherian ring, then $R[x]$ is as well.  
\end{corollary} 

\begin{proof} If $J$ is a pure ideal of $R[x]$, then by Theorem \ref{Theorem 5 polynomial}, $J=I[x]$ with $I$ a pure ideal of $R$. Since $I$ is finitely generated, its extension $I[x]$ is also finitely generated. 
\end{proof}

We proved that the pure ideals of the polynomial ring $R[x]$ are exactly the extended pure ideals of $R$. We show that the same is true for the formal power series ring $R[[x]]$. However, the previous approach is ineffective here; a new method of proof is required: 

\begin{theorem}\label{Theorem 6 series ring} For every commutative ring $R$, the pure ideals of $R[[x]]$ are exactly the extended pure ideals of $R$. The same is true for purely-prime ideals. In particular, $\pdim(R)=\pdim(R[[x]])$.
\end{theorem}

\begin{proof} If $I$ is a pure ideal of $R$ then $I^{e}$, the extension of $I$ under the natural ring map $R\rightarrow S=R[[x]]$, is  a pure ideal of $S$. Conversely, let $J$ be a pure ideal of $S$. We show that $J=I^{e}$ with $I=\epsilon(J)$ is a pure ideal of $R$ where $\epsilon: S\rightarrow R$ is the evaluation map $f(x)\mapsto f(0)$. Since the ring map $\epsilon$ is surjective, so $I=\epsilon(J)$ is a pure ideal of $R$. If $f\in J$ then $f=fg$ for some $g\in J$. We may write $g=b+xh$ with $b=g(0)=\epsilon(g)\in\epsilon(J)=I$ and $h\in S$. Then $f(1-xh)=fb$. Since $1-xh$ has constant term $1$, it is a unit of $S$ (recall that a given formal power series $P(x)\in S$ is invertible in $S$ if and only if $\epsilon(P)=P(0)$, the constant term of $P(x)$, is invertible in $R$). Thus $f=b(f(1-xh)^{-1})\in I^{e}$. This shows that $J\subseteq I^{e}$. To see the reverse inclusion $I^{e}\subseteq J$, it suffices to show that every generator $a\in I$ of $I^{e}$ is a member of $J$. That is, we show that $I\subseteq J$. Since $I$ is pure, so $a=ab$ for some $b\in I=\epsilon(J)$. Thus there is $g=b+xh\in J$. Then $ag=a(b+xh)=a(1+xh)$. Since $1+xh$ is invertible in $S$, we have $a=(1+xh)^{-1}ag\in J$. Therefore, $J=I^{e}$. \\
For any ideals $I$ and $J$ of $R$, it is easy to see that $I\subseteq J$ if and only if $I^{e}\subseteq J^{e}$. \\
Now, it is clear that an ideal $I$ of $R$ is purely-prime (resp. purely-maximal) if and only if $I^{e}$ is a purely-prime (resp. purely-maximal) ideal of $R[[x]]$. In particular, $\pdim(R)=\pdim(R[[x]])$.
\end{proof}

\begin{example}\label{Example 1 coefficientwise} We give an example of ring $R$ with a pure ideal $I$ such that the extended (pure) ideal $I^e$ of $R[[x]]$ is strictly contained in $I[[x]]$. This is done by proving that the ideal $I[[x]]$ is not pure. Consider the infinite direct product ring $R=\prod\limits_{m\geqslant0}k$ with $k$ a field. The direct sum ideal $I=\bigoplus\limits_{m\geqslant0}k$ of $R$ is pure, because it is generated by the idempotents $e_{m}=(\delta_{m,n})_{n\geqslant0}\in R$ where $\delta_{m,n}$ is the Kronecker delta. We have $I^{e}\subseteq I[[x]]$. To see that this inclusion is strict, it suffices to show that the ideal $I[[x]]$ is not pure. We have $f=\sum\limits_{m\geqslant0}e_{m}x^{m}\in I[[x]]$. If $I[[x]]$ is pure, then $f=fg$ for some $g=\sum\limits_{m\geqslant0}c_{m}x^{m}\in I[[x]]$ with $c_{m}=(a_{m,d})_{d\geqslant0}\in I$ for all $m\geqslant0$. For each $n\geqslant0$ consider the $n$th coordinate ring map $\phi_{n}:R[[x]]\rightarrow k[[x]]$ which is induced by the $n$th projection ring map $R\rightarrow k$,  $(r_{m})_{m\geqslant0}\mapsto r_{n}$. Then from $\phi_{n}(f)=\phi_{n}(fg)$ we get that $x^{n}=x^{n}(\sum\limits_{m\geqslant0}a_{m,n}x^{m})$. Since $k[[x]]$ is a domain (more generally, $x$ is a non-zero-divisor), we obtain that $\sum\limits_{m\geqslant0}a_{m,n}x^{m}=1$ and so $a_{0,n}=1$. This shows that $c_{0}=(a_{0,n})_{n\geqslant0}=(1,1,1,\ldots)\in I$ which is a contradiction. Hence, the ideal $I[[x]]$ is not pure. Therefore the inclusion $I^{e}\subset I[[x]]$ is strict.  
\end{example}

We show that Theorem \ref{Theorem 6 series ring} is indeed a special case of a general fact about split surjections and the Jacobson radical principle:  

\begin{theorem} Let $\phi: A\rightarrow B$ be a ring map such that $\Ker(\phi)\subseteq \mathrm{Jac}(A)$ and suppose $\phi$ admits a ring section, i.e., there is a ring map $s:B\rightarrow A$ with $\phi\circ s$ the identity map. Then the pure ideals of $A$ are exactly the extended pure ideals of $B$ under $s$. The same is true for purely-prime ideals. In particular, $\pdim(A)=\pdim(B)$. 
\end{theorem}

\begin{proof} Let $I$ be a pure ideal of $A$. Since $\phi$ is surjective, $J=\phi(I)$ is a pure ideal of $B$. We show that $I=J^{e}$ where $J^{e}=(s(b) : b\in J)$ is the extension of $J$ under $s$. The pure ideals $I$ and $J^{e}$ have the same image in $A/\Ker(\phi)$. Then by Lemma \ref{Lemma 12 Jacobson}, $I=J^{e}$. \\
For any ideals $I$ and $J$ of $B$, then $I\subseteq J$ if and only if $I^{e}\subseteq J^{e}$ where $I^{e}$ is the extension of $I$ under $s$. Now one can easily check that an ideal $I$ of $B$ is purely-prime (resp. purely-maximal) if and only if $I^{e}$ is a purely-prime (resp. purely-maximal) ideal of $A$. In particular, $\pdim(A)=\pdim(B)$.
\end{proof}

Using the finite direct products of the rings $R_{\Omega}$, we can construct new rings with various strictly ascending chains of purely-prime ideals. 

The constructions and theorems developed in this article establish a foundation for further investigation of pure ideals. In a subsequent work, we will characterize the structure of all prime ideals of the ring $R_{\Omega}$. Although the rings $R_{\Omega}$ are constructed in exactly the same way, we will also prove that they are pairwise non-isomorphic. We shall provide key examples showing that the strict inequality $\pdim(R)<\pdim(R/J)$ and, most importantly, the reverse inequality $\pdim(R)>\pdim(R/J)$ can hold for some rings with $J$ an ideal contained in the Jacobson radical. We will also give a very surprising and unexpected example showing that there is a commutative ring with a purely-prime ideal that is not representable by the pure part of a prime ideal. Finally, inspired by this example, we will characterize the purely-prime ideals of any ring that can be represented by the pure parts of prime ideals. \\

\textbf{Acknowledgments.} The author would like to express his deep gratitude to the scientist who kindly consulted the advanced AI models GPT-5.6 Pro and GPT-6 Astra on the questions raised in this article. In fact, without the help and assistance of these advanced AI models, the highly nuanced and technical discoveries in this article would have been nearly impossible. The author has carefully checked the mathematical arguments and proofs included in the article and assumes full responsibility for their correctness and presentation. The author also sincerely thanks Professor Brian Conrad for valuable discussions.

\end{document}